\documentclass[a4paper]{amsart}

\usepackage{amsmath,amssymb,amsthm}
\usepackage{indentfirst,color}
\usepackage[colorlinks,citecolor=red,linkcolor=blue,urlcolor=cyan]{hyperref}

\theoremstyle{plain}
\newtheorem{theorem}{Theorem}[section]
\newtheorem{proposition}[theorem]{Proposition}
\newtheorem{lemma}[theorem]{Lemma}
\newtheorem{corollary}[theorem]{Corollary}
\theoremstyle{definition}
\newtheorem{definition}[theorem]{Definition}
\theoremstyle{remark}
\newtheorem{remark}[theorem]{Remark}
\numberwithin{equation}{section}

\begin{document}

\title[Griffiths Positivity, Pluriharmonicity and Flatness]{Characterizations of Griffiths Positivity, Pluriharmonicity and Flatness}

\author{Zhuo Liu}
\address{Beijing Institute of Mathematical Sciences and Applications, Beijing 101408, China. Department of Mathematics and Yau Mathematical Sciences Center, Tsinghua University, Beijing 100084, China.}
\email{liuzhuo@amss.ac.cn; liuzhuo@bimsa.cn}

\author{Wang Xu}
\address{School of Mathematical Sciences, Peking University, Beijing 100871, China. Current Address: School of Mathematics, Sun Yat-sen University, Guangzhou 510275, China.}
\email{xuwang@amss.ac.cn}

\date{}

\begin{abstract}
Deng-Ning-Wang-Zhou showed that a Hermitian holomorphic vector bundle is Griffiths semi-positive if it satisfies the optimal $L^2$-extension condition. As a generalization, we present a quantitative characterization of Griffiths positivity in terms of certain $L^2$-extension conditions.
We also show that a $\mathbb{R}$-valued measurable function is pluriharmonic if and only if it satisfies the equality part of the optimal $L^p$-extension condition. This answers a conjecture of Inayama affirmatively.
Moreover, the flatness of a possibly singular Hermitian metric is also equivalent to the equality part of the optimal $L^p$-extension condition.
\end{abstract}

\keywords{Griffiths positivity, Pluriharmonicity, Flatness, Bergman kernel, Optimal $L^p$-extension, Holomorphic vector bundle, Singular Hermitian metric}

\subjclass[2020]{31C10, 32L15, 32A36, 32D15, 32L05, 32U40}

\maketitle

\section{Introduction}

Positivities, such as plurisubharmonicity and Griffiths/Nakano positivity, play fundamental roles in the study of several complex variables and complex geometry. These positivity concepts have led to numerous important results.
Psh (short for plurisubharmonic) functions are not necessarily smooth, this offers significant advantages in certain problems. As for vector bundles, there is a constant interest in exploring singular metrics with certain kind of positivity.

Let $E$ be a holomorphic vector bundle over a complex manifold $X$. A \textit{singular Hermitian metric} $h$ on $E$ is a measurable map from the base manifold $X$ to the space of non-negative Hermitian forms on the fibers, satisfying $0<\det h<+\infty$ almost everywhere.
It is well-known that, when $h$ is smooth, $h$ is Griffiths semi-positive if and only if $\log|u|_{h^*}$ is psh for any local holomorphic section $u$ of the dual bundle. This characterization naturally leads to a definition of Griffiths positivity for singular Hermitian metrics (see \cite{BP08,Raufi,PT18}), which has proven to be very useful.
In particular, when $E$ is a line bundle, the singular metric $h$ is Griffiths semi-positive if and only if the local weight $\varphi:=-\log h$ is psh.

Recall that, for psh functions and Nakano semi-positive Hermitian holomorphic vector bundles on pseudoconvex domains or Stein manifolds, there are H\"ormander's $L^2$ estimates for the $\overline{\partial}$-equations \cite{Hor65} and Ohsawa-Takegoshi's $L^2$ extension theorem \cite{OT87}.
Since the publication of \cite{OT87}, there has been considerable interests in refining the estimate in Ohsawa-Takegoshi's $L^2$ extension theorem. After the breakthrough of Guan-Zhou-Zhu \cite{GZZ11}, in 2012, B{\l}ocki \cite{Blocki13} and Guan-Zhou \cite{GuanZhou12} successfully obtained the optimal $L^2$ extension theorem.

In \cite{DWZZ18,DNW21,DNWZ22}, Deng, Ning, Wang, Zhang, Zhou established the converse $L^2$ theory by giving alternative characterizations of plurisubharmonicity and Griffiths/Nakano positivity in terms of various $L^2$-conditions for $\overline{\partial}$.
They proved that a smooth Hermitian metric is Nakano semi-positive if and only if it satisfies the ``\textit{optimal $L^2$-estimate condition}''. This characterization leads to a definition of Nakano positivity for singular Hermitian metrics (see \cite{Ina-AG}) and provides a positive answer to a question of Lempert (see \cite{LiuYangZhou21}).
Moreover, if $h$ is a singular Hermitian metric on $E$ such that $|u|_{h^*}$ is upper semi-continuous for any local holomorphic section $u$ of $E^*$, then $(E,h)$ is Griffiths semi-positive if it satisfies the ``\textit{optimal $L^2$-extension condition}'' (see Definition \ref{Def:OptLpExt} and Theorem \ref{Thm:DNWZ}); the converse is also true if $\dim X=1$ or $\operatorname{rank}E=1$.
For more results on characterizations of positivity, we refer the readers to \cite[etc.]{DWZZ18, HosonoInayama, DNW21, DNWZ22, DengZhang21, KharePingali, Watanabe}.

In B{\l}ocki's and Guan-Zhou's optimal $L^2$ extension theorem, it is worth noting that they only required \textit{semi}-positive curvature. The term ``optimal'' means that the uniform estimate provided in the theorem cannot be improved within the considered setting.
Provided \textit{strictly} positive curvature, the setting becomes narrower, and it is not surprised that ``sharper'' estimates can be obtained (see Hosono \cite{Hosono19}, Kikuchi \cite{Kikuchi}, Xu-Zhou \cite{XuZhou22}).
Compared this with the result in \cite{DNWZ22}, it suggests that $L^2$ extensions with sharper estimate would imply strictly Griffiths positivity. 

In \cite{Inayama}, Inayama introduced the notion of ``\textit{$L^2$-extension index}'' for smooth functions (resp. Hermitian holomorphic vector bundles) over planar domains and gave quantitative estimates of the complex Hessian (resp. the Chern curvatures) by using these indexes.
Recall that, a \textit{holomorphic cylinder} in $\mathbb{C}^n$ is a domain of the form $P_{A,r,s}:=A(\mathbb{D}_r\times\mathbb{B}_s^{n-1})$, where $A\in \textbf{U}(n)$ is unitary and $r,s>0$ (see \cite{DNW21,DNWZ22}). For convenience, we define the ``diameter'' of $P_{A,r,s}$ to be
\begin{equation}
\mathfrak{d}(P_{A,r,s}):=\sqrt{\tfrac{1}{2}r^2+\tfrac{n-1}{n}s^2}.
\end{equation}
The first result of this paper is a quantitative characterization of Griffiths positivity in terms of certain $L^2$-extension conditions, which generalizes Inayama's result to higher dimensions. Notice that, we don't require $c\geqslant0$ and our proof is different.

\begin{theorem}\label{MainThm1}
Let $(E,h)$ be a holomorphic vector bundle over a domain $\Omega\subset\mathbb{C}^n$, equipped with a smooth Hermitian metric. Then the following are equivalent:
\begin{itemize}
	\item[(i)] $i\Theta_h \geqslant_\textup{Grif} c\omega\otimes\textup{Id}_E$ at $x\in\Omega$, where $c\in\mathbb{R}$;
	\item[(ii)] for any $\varepsilon>0$, there is a constant $\delta>0$ such that for any $\xi\in E_x$ and any holomorphic cylinder $x+P\subset\Omega$ with $\mathfrak{d}(P)<\delta$, there exists a holomorphic section $f\in\Gamma(x+P,E)$ satisfying $f(x)=\xi$ and
	\begin{equation}
		\frac{1}{\textup{Vol}(P)}\int_{x+P}|f|_h^2d\lambda \leqslant \big(1-(c-\varepsilon)\mathfrak{d}(P)^2\big) |\xi|_h^2.
	\end{equation}
\end{itemize}
\end{theorem}

Condition (ii) with $c=0$, which is equivalent to the Griffiths semi-positivity at the given point, appears to be weaker than the optimal $L^2$-extension condition. However, Prof. Fusheng Deng and Prof. Zhiwei Wang expected that the optimal $L^2$-extension condition would also be equivalent to the Griffiths semi-positivity.

Let $\varphi$ be an upper semi-continuous function on a domain $\Omega\subset\mathbb{C}^n$, for any holomorphic cylinder $x+P\subset\Omega$, the \textit{$L^2$-extension index} of $\varphi$ is defined as
\begin{equation}
	L_\varphi(x,P) := \inf\left\{ \frac{\int_{x+P}|f|^2e^{-\varphi}d\lambda}{\textup{Vol}(P)e^{-\varphi(x)}}: f\in\mathcal{O}(x+P),f(x)=1 \right\}.
\end{equation}
By Montel's theorem, the infimum in the above definition is achievable. According to \cite{Blocki13,GuanZhou15} and \cite{DNW21}, $\varphi$ is psh if and only if $L_\varphi\leqslant1$. Moreover, by Theorem \ref{MainThm1}, if $\varphi$ is smooth and strictly psh near $x$, then $L_\varphi(x,P)<1$ for some holomorphic cylinder $x+P\Subset\Omega$. Having these observations, it is natural to ask whether $L_\varphi(x,P)\equiv1$ implies $i\partial\bar{\partial}\varphi\equiv0$?

Addressing the above question, Inayama \cite{Inayama} proved that a \textbf{smooth} function $\varphi$ on $\Omega\subset\mathbb{C}^n$ is pluriharmonic if and only if $L_\varphi(x,P)\equiv1$. He conjectured that such a characterization would also hold for \textbf{upper semi-continuous} functions (see \cite[Conjecture A.2]{Inayama}).
Inayama communicated his conjecture via email to the second author on May 12, 2023. A few days later, he also provided us with a proof for \textbf{continuous} $\varphi$.

\begin{theorem}[see \cite{InayamaNote}] \label{MainThm:Inayama}
Let $\varphi$ be a continuous function on  $\Omega\subset\mathbb{C}^n$, then $\varphi$ is pluriharmonic if and only if $L_\varphi(x,P)=1$ for all holomorphic cylinder $x+P\subset\Omega$.
\end{theorem}

Inayama's proof goes as follows: firstly, $L_\varphi\leqslant1$ means that $\varphi$ satisfies the optimal $L^2$-extension condition, then $\varphi$ is psh; for any polynomial $q$ and any holomorphic cylinder $x+P$, $L_\varphi(x,P)\geqslant1$ yields $\int_{x+P}|e^q|^2e^{-\varphi}d\lambda \geqslant \textup{Vol}(P)|e^{q(x)}|^2e^{-\varphi(x)}$. Since $\varphi$ is lower semi-continuous, it follows that $e^{-\varphi+2\operatorname{Re}q}$ is psh for any polynomial $q$. Consequently, $-\varphi$ is also psh. 

In this paper, we shall prove that the continuity assumption in Theorem \ref{MainThm:Inayama} is superfluous, i.e. the continuity of $\varphi$ follows from the condition $L_\varphi\equiv1$. Moreover, the $L^2$-extension condition can be replaced by a similar $L^p$-extension condition.

\begin{theorem}\label{MainThm2}
Let $\varphi:\Omega\to\mathbb{R}$ be a measurable function on a domain $\Omega\subset\mathbb{C}^n$ and $p>0$ be a constant, then the following conditions are equivalent:\par
\begin{itemize}
	\item[(i)] $\varphi$ is pluriharmonic on $\Omega$;
	\item[(ii)] for any holomorphic cylinder $x+P\subset\Omega$,
	\begin{equation} \label{Cond1}
		\inf\left\{ \int_{x+P}|f|^pe^{-\varphi}d\lambda: f\in\mathcal{O}(x+P),f(x)=1 \right\} = \textup{Vol}(P)e^{-\varphi(x)}.
	\end{equation}
	\item[(iii)] there exists a positive continuous function $\gamma\ll1$ on $\Omega$ such that \eqref{Cond1} holds for any holomorphic cylinder $x+P\Subset\Omega$ with $\mathfrak{d}(P)<\gamma(x)$.
\end{itemize}
\end{theorem}

The main idea of the proof is that \eqref{Cond1} implies the lower/upper semi-continuity of $\varphi$. At the beginning, we don't know whether the infimum in \eqref{Cond1} is achievable. This is one of the difficulties in our proof. Notice that, the assumption that $\varphi$ is $\mathbb{R}$-valued is necessary (see Remark \ref{Rmk:NoInf}).

In the case of $n=1$ and $p=2$, assuming $\varphi$ is subharmonic, one can prove a stronger result: if \eqref{Cond1} holds for a single disc $x+\mathbb{D}_r$, then $\varphi$ is harmonic on $x+\mathbb{D}_r$. This result is a consequence of Theorem 1.11 of Guan-Mi \cite{GuanMi22}, and we will give a shorter proof by using Corollary 1.5 of \cite{GuanMi22}.

\begin{theorem}\label{MainThm3}
	Let $\varphi>-\infty$ be a subharmonic function on $\mathbb{D}$, then $\varphi$ is harmonic on $\mathbb{D}$ if and only if
	$$ \pi B_{\mathbb{D}}(0;e^{-\varphi}) = e^{\varphi(0)}. $$
\end{theorem}

Here, $B_\mathbb{D}(\cdot;e^{-\varphi})$ denotes the weighted Bergman kernel of $\mathbb{D}$, i.e.
$$ B_{\mathbb{D}}(0;e^{-\varphi}) := \sup\left\{ |f(0)|^2: f\in\mathcal{O}(\mathbb{D}), \int_\mathbb{D}|f|^2e^{-\varphi}d\lambda\leqslant1 \right\}. $$
Since $\varphi$ is subharmonic, by the optimal $L^2$ extension theorem, $\pi B_{\mathbb{D}}(0;e^{-\varphi}) \geqslant e^{\varphi(0)}$. The above theorem shows that the equality holds if and only if $\varphi$ is harmonic.

This result is similar to Suita's conjecture: let $\Omega$ be an open Riemann surface admitting nontrivial Green's function, then $\pi B_\Omega(x)\geqslant c_\beta(x)^2$, and the equality holds if and only if $\Omega$ is conformally equivalent to the unit disc $\mathbb{D}$ less a possible closed polar set. Here, $c_\beta$ denotes the logarithmic capacity of $\Omega$.
The inequality part of Suita's conjecture was proved by B{\l}ocki \cite{Blocki13} and Guan-Zhou \cite{GuanZhou12}, and the equality part was proved by Guan-Zhou \cite{GuanZhou15}.
In short, Theorem \ref{MainThm3} characterizes the weight and Suita's conjecture characterizes the base. 

Theorem \ref{MainThm2} can also be generalized to the case of holomorphic vector bundles. Let $(E,h)$ be a holomorphic vector bundle equipped with a \textbf{smooth} Hermitian metric, Inayama \cite{Inayama} showed that $(E,h)$ is curvature flat (i.e. $\Theta_h\equiv0$) if and only if $(E,h)$ satisfies the equality part of the optimal $L^2$-extension condition. In this paper, we show that the smoothness assumption in this characterization can also be dropped and the ``optimal $L^2$-extension condition'' can be replaced by the ``optimal $L^p$-extension condition''.

\begin{theorem}\label{MainThm4}
Let $E$ be a holomorphic vector bundle over a domain $\Omega\subset\mathbb{C}^n$ and $p>0$ be a constant. Let $h$ be a singular Hermitian metric on $E$ such that $0<\det h<+\infty$ everywhere, then the following conditions are equivalent:\par
\begin{itemize}
	\item[(i)] $h$ is smooth and $\Theta_h\equiv0$;
	\item[(ii)] for any holomorphic cylinder $x+P\subset\Omega$ and any $v\in E_x$,
	\begin{equation}\label{Cond2}
		\inf\left\{ \int_{x+P}|f|_h^pd\lambda: f\in\Gamma(x+P,E), f(x)=v \right\} = \textup{Vol}(P)|v|_h^p.
	\end{equation}
	\item[(iii)] there exists a positive continuous function $\gamma\ll1$ on $\Omega$ such that \eqref{Cond2} holds for any holomorphic cylinder $x+P\Subset\Omega$ with $\mathfrak{d}(P)<\gamma(x)$ and any $v\in E_x$.
\end{itemize}
\end{theorem}

As before, at the beginning, we don't know whether the infimum in \eqref{Cond1} is achievable.
Another technical problem arises when we considering the Chern curvature of a singular Hermitian metric.
When $E$ is a line bundle, the Chern curvature current $i\Theta_h := i\overline{\partial}(h^{-1}\partial h)$ is well-defined as long as $\log h\in L_\textup{loc}^1(X)$.
However, Raufi's \cite{Raufi} example showed that defining the curvature is not possible in general.
In \cite{Raufi} and \cite{Ina-JGA}, with some additional regularity conditions, it was proved that the Chern curvature current of a Griffiths semi-positive/negative singular Hermitian metric has measure coefficients.

Our strategy is as follows: firstly, we show that \eqref{Cond2} implies the continuity of $h$; subsequently, $\Theta_h:=\overline{\partial}(h^{-1}\partial h)$ is well-defined and vanishes in the sense of currents; finally, we show that $h$ is smooth and flat. Notice that, $(E,h)$ is curvature flat if and only if there exists local unitary holomorphic frame field (see Lemma \ref{Lemma:FlatUnitary}).

The remaining parts of this article are organized as follows. In section \ref{Sec:Pre}, we recall some preparatory results. In section \ref{Sec:Quan}, we prove a quantitative characterization of Griffiths positivity. In section \ref{Sec:PH}, we prove two characterizations of pluriharmonic functions. In section \ref{Sec:Flat}, we prove a characterization of flatness. In the appendix, we study the regularity of the infimum that appeared in \eqref{Cond1}.

\section{Preliminaries} \label{Sec:Pre}

In this paper, $d\lambda$ and $\omega:=i\partial\bar{\partial}|z|^2$ always denote the Lebesgue measure and the standard K\"ahler form of $\mathbb{C}^n$. For any $a\in\mathbb{C}$, $x\in\mathbb{C}^n$ and $r,s\in\mathbb{R}_+$, we denote $\mathbb{D}(a;r):=\{\tau\in\mathbb{C}:|\tau-a|<r\}$ and $\mathbb{B}^n(x;s):=\{z\in\mathbb{C}^n:|z-x|<s\}$. For simplicity, $\mathbb{D}_r:=\mathbb{D}(0;r)$ and $\mathbb{B}_s^n:=\mathbb{B}^n(0;s)$.

A \textit{holomorphic cylinder} in $\mathbb{C}^n$ is a domain of the form $P_{A,r,s}:=A(\mathbb{D}_r\times\mathbb{B}_s^{n-1})$, where $A\in \textbf{U}(n)$ and $r,s>0$. It is well-known that psh functions satisfy the mean value inequality on holomorphic cylinders. Conversely, this property characterizes all psh functions.

\begin{lemma}[see \cite{DNW21}] \label{Lemma:CharPsh}
Let $\varphi$ be an upper semi-continuous function on a domain $\Omega\subset\mathbb{C}^n$, then $\varphi$ is psh if and only if
\begin{equation} \label{Eq:MeanValIneq}
	\varphi(x) \leqslant \frac{1}{\textup{Vol}(P)} \int_{x+P}\varphi d\lambda
\end{equation}
 for any $x\in\Omega$ and any sufficiently small holomorphic cylinder $x+P\subset\Omega$.
\end{lemma}

\begin{corollary} \label{Cor:QuanPsh}
Let $\varphi$ be a $C^2$-smooth real-valued function defined in a neighborhood of $0\in\mathbb{C}^n$. If there is a constant $c\in\mathbb{R}$ so that
\begin{equation*}
\varphi(0) \leqslant \frac{1}{\textup{Vol}(P)} \int_{P}\left(\varphi(z)-c|z|^2\right) d\lambda_z
\end{equation*}
for any sufficiently small holomorphic cylinder $P$, then $i\partial\bar{\partial}\varphi\geqslant c\omega$ at $0$.
\end{corollary}

We need an explicit formula for $\int_P|z|^2d\lambda_z$. Let $P_{A,r,s}$ be a holomorphic cylinder in $\mathbb{C}^n$. Since $|z|^2$ is invariant under any unitary transformation,
\begin{align*}
\int_{P_{A,r,s}} |z|^2 d\lambda_z = &\, \int_{\mathbb{D}_r\times\mathbb{B}_s^{n-1}} |z|^2 d\lambda_z = \int_{\mathbb{D}_r\times\mathbb{B}_s^{n-1}} (|z_1|^2+|z'|^2) d\lambda_z \\
= &\, \textup{Vol}(\mathbb{B}_s^{n-1}) \int_{\mathbb{D}_r} |z_1|^2d\lambda_{z_1} + \textup{Vol}(\mathbb{D}_r) \int_{\mathbb{B}_s^{n-1}} |z'|^2d\lambda_{z'} \\
= &\, (\tfrac{1}{2}r^2+\tfrac{n-1}{n}s^2)\textup{Vol}(P_{A,r,s}).
\end{align*}
For convenience, we define the ``diameter'' of $P_{A,r,s}$ to be
\begin{equation} \label{Eq:Diam}
\mathfrak{d}(P_{A,r,s}):=\sqrt{\tfrac{1}{2}r^2+\tfrac{n-1}{n}s^2},	
\end{equation}
then
\begin{equation} \label{Eq:MeanInt}
\frac{1}{\textup{Vol}(P_{A,r,s})} \int_{P_{A,r,s}}|z|^2d\lambda_z = \mathfrak{d}(P_{A,r,s})^2.
\end{equation}

It is well-known that a Hermitian holomorphic vector bundle $(E,h)$ is Griffiths semi-negative if and only if $\log|u|_h$ is psh for any local holomorphic section $u$ of $E$. We have the following quantitative version of this fact.

\begin{lemma} \label{Lemma:QuanGrif}
Let $(E,h)$ be a holomorphic vector bundle over a domain $\Omega\subset\mathbb{C}^n$, equipped with a smooth Hermitian metric. Let $c\in\mathbb{R}$ be a constant, then $i\Theta_h \leqslant_\textup{Grif} -c\omega\otimes\textup{Id}_E$ at $x\in \Omega$ if and only if $i\partial\bar{\partial}\log|u|_h^2 \geqslant c\omega$ at $x$ for any local holomorphic section $u$ of $E$ with $u(x)\neq 0$. 
\end{lemma}

\begin{proof}
We choose a holomorphic frame $(e_1,\ldots,e_r)$ of $E$ in some neighborhood of $x$ such that $h_{\alpha\overline{\beta}}(x) = \delta_{\alpha\beta}$ and $dh_{\alpha\overline{\beta}}(x)=0$, where $h_{\alpha\overline{\beta}} := \langle{e_\alpha,e_\beta}\rangle_h$. Then the components of $h\Theta_h = h\overline{\partial}(h^{-1}\partial h)$ are
\begin{equation*}
	R_{i\overline{j}\alpha\overline{\beta}} = -\frac{\partial^2 h_{\alpha\overline{\beta}}}{\partial z_i\partial\overline{z}_j}
	+ \sum_{\alpha',\beta'} h^{\alpha'\overline{\beta'}} \frac{\partial h_{\alpha\overline{\beta'}}}{\partial z_i} \frac{\partial h_{\alpha'\overline{\beta}}}{\partial\overline{z}_j}
	\overset{(\text{at }x)}{=} -\frac{\partial^2 h_{\alpha\overline{\beta}}}{\partial z_i\partial\overline{z}_j}(x).
\end{equation*}

Let $u=\sum_\alpha u_\alpha e_\alpha$ be any local holomorphic section of $E$ with $u(x)\neq 0$, then
\begin{multline*}
	\frac{\partial^2}{\partial z_i\partial\overline{z}_j}|_x (\log|u|_h^2) = -\frac{1}{|u(x)|_h^2} \sum_{\alpha,\beta}R_{i\overline{j}\alpha\overline{\beta}}(x)u_\alpha(x)\overline{u_\beta(x)} \\
	+ \frac{1}{|u(x)|_h^4} \left(\sum_\alpha\frac{\partial u_\alpha}{\partial z_i}\overline{\frac{\partial u_\alpha}{\partial z_j}} \sum_\beta u_\beta\overline{u_\beta} - \sum_\alpha\frac{\partial u_\alpha}{\partial z_i}\overline{u_\alpha} \sum_\beta u_\beta\overline{\frac{\partial u_\beta}{\partial z_j}}\right).
\end{multline*}

If $i\Theta_h \leqslant_\textup{Grif} -c\omega\otimes\textup{Id}_E$ at $x$, then
\begin{equation*}
	-\sum_{i,j,\alpha,\beta} R_{i\overline{j}\alpha\overline{\beta}}(x) u_\alpha(x)\overline{u_\beta(x)}a_i\overline{a_j} \geqslant c|u(x)|_h^2|a|^2, \quad \forall a\in\mathbb{C}^n.
\end{equation*}
By the Cauchy-Schwarz inequality, the second term on the right hand side is always positive semi-definite. Therefore,
\begin{equation*}
	i\partial\bar{\partial}\log|u|_h^2 \geqslant c\omega \text{ at }x.
\end{equation*}

Conversely, for any non-zero vectors $a=(a_1,\ldots,a_n)\in\mathbb{C}^n$ and $\xi=(\xi_1,\ldots,\xi_r)\in\mathbb{C}^r$, we can define a local holomorphic section $u$ of $E$ by $u(z)=\sum_\alpha \xi_\alpha e_\alpha(z)$. If $i\partial\bar{\partial}\log|u|_h^2\geqslant c\omega$ at $x$, then
\begin{equation*}
	-\sum_{i,j,\alpha,\beta} R_{i\overline{j}\alpha\overline{\beta}}(x) a_i\overline{a_j}\xi_\alpha\overline{\xi_\beta} = \sum_\alpha|\xi_\alpha|^2 \sum_{i,j}\frac{\partial^2}{\partial z_i\partial\overline{z}_j}|_x(\log|u|_h^2) a_i\overline{a_j} \geqslant c|a|^2|\xi|^2.
\end{equation*}
This completes the proof.
\end{proof}

\begin{remark}
Let $E\to X$ be a holomorphic vector bundle equipped with a smooth Hermitian metric $h$ and let $\rho$ be a continuous real $(1,1)$-form on $X$.
By similar arguments as Lemma \ref{Lemma:QuanGrif}, we can prove that $i\Theta_h \leqslant_\textup{Grif} -\rho\otimes\textup{Id}_E$ at $x\in X$ if and only if $i\partial\bar{\partial}\log|u|_h^2 \geqslant \rho$ at $x$ for any local holomorphic section $u$ of $E$ with $u(x)\neq 0$.
\end{remark}

\begin{definition}[{see \cite{BP08,Raufi,PT18}}] \label{Def:GrifPos}
Let $E$ be a holomorphic vector bundle and $h$ be a singular Hermitian metric on $E$. We say that $h$ is \emph{Griffiths semi-negative} if $\log|u|_h$ is psh for any local holomorphic section $u$ of $E$. We say that $h$ is \emph{Griffiths semi-positive} if the dual metric $h^*$ on $E^*$ is Griffiths semi-negative.
\end{definition}

\begin{proposition}[see \cite{Raufi}]\label{Prop:EquivGrif}
Let $(E,h)$ be a holomorphic vector bundle equipped with a singular Hermitian metric and $p>0$ be a constant, then the following conditions are equivalent:
\begin{enumerate}
	\item[(i)] $(E,h)$ is Griffiths semi-negative;
	\item[(ii)] $\log|u|_h$ is psh for any local holomorphic section $u$ of $E$;
	\item[(iii)] $|u|_h^p$ is psh for any local holomorphic section $u$ of $E$;
	\item[(iv)] $(E^*,h^*)$ is Griffiths semi-positive.
\end{enumerate} 
\end{proposition}

\begin{proof}
We only need to show that (iii) implies (ii).
Recall that, for a non-negative function $v$, $\log v$ is psh if and only if $ve^{p\operatorname{Re}g}$ is psh for every holomorphic polynomial $g$. Let $u$ be a local holomorphic section of $E$, the condition (iii) says that $|ue^g|_h^p=|u|_h^pe^{p\operatorname{Re}g}$ is psh for any holomorphic polynomial $g$. Therefore, $\log|u|_h^p$ is psh.
\end{proof}

\begin{lemma}[see \cite{BP08,PT18}] \label{Lemma:PT}
Let $E$ be a holomorphic vector bundle over a domain $\Omega\subset\mathbb{C}^n$ and $h$ be a Griffiths semi-negative singular Hermitian metric on $E$. Assume that $U\Subset V\Subset\Omega$ are open subsets such that $E|_V$ is trivial. Then

$(1)$ There is a sequence of Griffiths negative smooth Hermitian metrics $\{h_\nu\}_{\nu=1}^\infty$ on $E|_U$ decreasing pointwise to $h$. In particular, $\log \det h$ is a psh function.

$(2)$ There exists a constant $C_U>0$ such that $C_U^{-1}(\det h)I_r \leqslant h \leqslant C_U I_r$ on $U$, where $r=\operatorname{rank}E$ and $I_r$ is the $r\times r$ identity matrix.
\end{lemma}

In \cite{Raufi}, Raufi constructed an example showing that the formal Chern curvature current $\Theta_h:=\overline{\partial}(h^{-1}\partial h)$ of a Griffiths semi-negative singular Hermitian metric $h$ may not have measure coefficients. Even so, this is still possible under some additional regularity conditions.

\begin{lemma}[see \cite{Raufi}] \label{Lemma:Raufi}
Let $(E,h)$ be a holomorphic vector bundle equipped with a Griffiths semi-negative \textbf{continuous} Hermitian metric. Let $\{h_\nu\}_{\nu=1}^\infty$ be any sequence of Griffiths negative smooth Hermitian metrics decreasing pointwise to $h$.
Then, in any local trivialization of $E$, the entries of $\partial h$ are $L_\textup{loc}^2$-forms, the entries of the Chern curvature $\Theta_h := \overline{\partial}(h^{-1}\partial h)$ are currents with measure coefficients, and $\Theta_{h_\nu} := \overline{\partial}(h_\nu^{-1}\partial h_\nu)$ converge weakly to $\Theta_h$ as currents with measure coefficients.  
\end{lemma}

Recall that, for psh functions on pseudoconvex domains, B{\l}ocki \cite{Blocki13} and Guan-Zhou \cite{GuanZhou12,GuanZhou15} proved several $L^2$ extension theorems with optimal estimates. Using Berndtsson-P\u{a}un's iterative method, Guan-Zhou \cite{GuanZhou15} also obtained an optimal $L^p$ extension theorem, where $0<p<2$. In this paper, we only need the following special case.

\begin{theorem}[Optimal $L^p$ extension theorem {\cite{GuanZhou15}}] \label{Thm:OptLp}
Let $\varphi$ be a psh function on a holomorphic cylinder $P$ and $0<p\leqslant 2$ be a constant. If $\varphi(0)\neq-\infty$, then there exists a holomorphic function $f\in\mathcal{O}(P)$ such that $f(0)=1$ and
$$ \int_P |f|^pe^{-\varphi} d\lambda \leqslant \textup{Vol}(P)e^{-\varphi(0)}. $$
\end{theorem}

\begin{proof}
For the convenience of readers, we recall the proof for $p\in(0,2)$.
By standard approximation procedures, we may assume that $\varphi$ is smooth and defined in a neighborhood of $\overline{P}$, then there exists an $f_1\in\mathcal{O}(P)$ with $f_1(0)=1$ and
$$ C := \int_P |f_1|^pe^{-\varphi} d\lambda < +\infty. $$
Since $\varphi+(2-p)\log|f_1|$ is a psh function, by the optimal $L^2$ extension theorem, there exists an $f_2\in\mathcal{O}(P)$ such that $f_2(0)=1$ and
$$ \int_P|f_2|^2e^{-\varphi-(2-p)\log|f_1|}d\lambda \leqslant \textup{Vol}(P)e^{-\varphi(0)}. $$
By H\"older's inequality,
\begin{align*}
\int_P|f_2|^pe^{-\varphi}d\lambda & = \int_P \big(|f_1|^pe^{-\varphi}\big)^{\frac{2-p}{2}} \big(|f_2|^2e^{-\varphi-(2-p)\log|f_1|}\big)^{\frac{p}{2}} d\lambda \\
& \leqslant \left(\int_P|f_1|^pe^{-\varphi}d\lambda\right)^{\frac{2-p}{2}} \left(\int_P|f_2|^2e^{-\varphi-(2-p)\log|f_1|}d\lambda\right)^{\frac{p}{2}}  \\
& \leqslant C^{\frac{2-p}{2}} \left(\textup{Vol}(P)e^{-\varphi(0)}\right)^{\frac{p}{2}}.
\end{align*}
We repeat the same procedure and get a sequence of holomorphic functions $\{f_k\}_{k=1}^\infty$ on $P$ such that $f_k(0)=1$ and
\begin{align*}
\int_P|f_{k+1}|^pe^{-\varphi}d\lambda & \leqslant \left(\int_P|f_k|^pe^{-\varphi}d\lambda\right)^{\frac{2-p}{2}} \left(\textup{Vol}(P)e^{-\varphi(0)}\right)^{\frac{p}{2}} \\
& \leqslant \cdots \leqslant C^{(\frac{2-p}{2})^k} \left(\textup{Vol}(P)e^{-\varphi(0)}\right)^{1-(\frac{2-p}{2})^k}.
\end{align*}
Applying Montel's theorem, we obtain a holomorphic function $f\in\mathcal{O}(P)$ such that $f(0)=1$ and $\int_P|f|^pe^{-\varphi}d\lambda \leqslant \textup{Vol}(P)e^{-\varphi(0)}$.
\end{proof}

\begin{definition}[see \cite{DNW21,DNWZ22}] \label{Def:OptLpExt}
(1) Let $\varphi$ be an upper semi-continuous function on a domain $\Omega\subset\mathbb{C}^n$. Let $p>0$ be a constant, we say that $\varphi$ satisfies the \emph{optimal $L^p$-extension condition} if for any $x\in\Omega$ with $\varphi(x)\neq-\infty$ and any holomorphic cylinder $x+P\subset\Omega$, there exists a holomorphic function $f\in\mathcal{O}(x+P)$ such that $f(x)=1$ and $\int_{x+P}|f|^pe^{-\varphi}d\lambda \leqslant \textup{Vol}(P)e^{-\varphi(x)}$.

(2) Let $E$ be a holomorphic vector bundle over a domain $\Omega\subset\mathbb{C}^n$, equipped with a singular Hermitian metric $h$. Let $p>0$ be a constant, we say that $(E,h)$ satisfies the \emph{optimal $L^p$-extension condition} if for any $x\in\Omega$, any $v\in E_x$ with $|v|_h=1$ and any holomorphic cylinder $x+P\subset\Omega$, there exists a holomorphic section $F\in\Gamma(x+P,E)$ such that $F(x)=v$ and $\int_{x+P}|F|_h^pd\lambda \leqslant \textup{Vol}(P)$.
\end{definition}

Theorem \ref{Thm:OptLp} says that psh functions satisfy the optimal $L^p$-extension condition $(0<p\leqslant2)$. Deng-Ning-Wang \cite{DNW21} showed that the converse is also true. The idea of their proof goes back to Guan-Zhou's \cite{GuanZhou15} approach to Berndtsson's theorem on the plurisubharmonic variation of Bergman kernels. Similarly, the optimal $L^p$-extension condition also implies the Griffiths semi-positivity \cite{DNWZ22}.

\begin{theorem}[see \cite{DNW21,DNWZ22}] \label{Thm:DNWZ}
$(1)$ Let $\varphi$ be an upper semi-continuous function on a domain $\Omega\subset\mathbb{C}^n$. If $\varphi$ satisfies the optimal $L^p$-extension condition for some $p>0$, then $\varphi$ is a psh function.

$(2)$ Let $E$ be a holomorphic vector bundle over a domain $\Omega\subset\mathbb{C}^n$ and let $h$ be a singular Hermitian metric on $E$ such that $|u|_{h^*}$ is upper semi-continuous for any local holomorphic section $u$ of $E^*$. If $(E,h)$ satisfies the optimal $L^p$-extension condition for some $p>0$, then $(E,h)$ is Griffiths semi-positive.
\end{theorem}

Finally, we recall some results that will be used in the subsequent sections.

\begin{lemma}[{\cite[Lemma 4.34]{GuanZhou15}}]\label{Lemma:GZ}
Let $\varphi\not\equiv-\infty$ be a psh function on a domain $\Omega\subset\mathbb{C}^n$ and $p>0$ be a constant. Let $\{f_j\}_{j=1}^\infty$ be a sequence of holomorphic functions on $\Omega$ such that $\sup_j \int_\Omega|f_j|^pe^{\varphi}d\lambda < +\infty$. Then $\{f_j\}_{j=1}^\infty$ is uniformly bounded on any compact subset of $\Omega$.
\end{lemma}

\begin{lemma} \label{Lemma:FlatUnitary}
Let $(E,h)$ be a holomorphic vector bundle over a domain $\Omega\subset\mathbb{C}^n$, equipped with a smooth Hermitian metric. Then $\Theta_h\equiv0$ if and only if there exists a unitary holomorphic frame field of $(E,h)$ on any holomorphic cylinder $x+P\subset\Omega$.
\end{lemma}

\begin{proof}
If $e=(e_1,\ldots,e_r)$ is a unitary holomorphic frame field of $(E,h)$ on $x+P$, then $h=(h_{\alpha\overline{\beta}})=I_r$ and $\Theta_h=\overline{\partial}(h^{-1}\partial h)=0$ on $x+P$.

Conversely, we assume that $\Theta_h\equiv0$. We fix a holomorphic cylinder $x+P\subset\Omega$. Since $x+P$ is contractible, the vector bundle $E|_{x+P}$ is topologically trivial. Since $x+P$ is Stein, by the Oka-Grauert principle (see \cite[Theorem 5.3.1]{Forstneric}), $E|_{x+P}$ is also holomorphically trivial.

We begin with an arbitrary holomorphic frame field $e=(e_1,\ldots,e_r)$ of $E|_{x+P}$ and try to find a holomorphic map $g:x+P\to\textup{GL}(r;\mathbb{C})$ such that the new frame $\tilde{e}:=eg$ is unitary. The following proof is a modification of \cite[\S 1.2]{KobVectBun}.

We consider a differential equation with unknown $g:x+P\to\mathbb{C}^{r\times r}$,
\begin{equation}\label{DiffEq}
	dg+(h^{-1}\partial h)g=0.
\end{equation}
The integrability condition is
$$ 0 = d(h^{-1}\partial h)g - (h^{-1}\partial h)dg = d(h^{-1}\partial h)g + (h^{-1}\partial h)(h^{-1}\partial h)g = \overline{\partial}(h^{-1}\partial h)g, $$
which is equivalent to $\Theta_h=0$. We choose an $a\in\text{GL}(r;\mathbb{C})$ so that $a^*h(x)a=I_r$. Here, $^*$ denotes the conjugate transpose of a matrix. By the Frobenius theorem, there exists a smooth solution $g(z)$ of \eqref{DiffEq} satisfying $g(x)=a$. The equation \eqref{DiffEq} can be decomposed as
$$ \overline{\partial}g=0, \quad \partial g+(h^{-1}\partial h)g=0. $$
In particular, $g:x+P\to\mathbb{C}^{r\times r}$ is holomorphic.

We consider a tube $(\tilde{e}_1,\ldots,\tilde{e}_r)$ of holomorphic sections defined by $\tilde{e}=eg$.
Let $\tilde{h}_{\alpha\overline{\beta}} = \langle{\tilde{e}_\alpha,\tilde{e}_\beta}\rangle_h$, then $\tilde{h}=(\tilde{h}_{\alpha\overline{\beta}})=g^*hg$. Clearly,
$$ \partial\tilde{h} = g^*(\partial h)g + g^*h\partial g = g^*h\big((h^{-1}\partial h)g + \partial g\big) = 0. $$
As $\overline{\partial}\tilde{h}=\overline{\partial\tilde{h}^t}= 0$, we know $d\tilde{h}=0$. Since $\tilde{h}(x)=a^*h(x)a=I_r$, it follows that $\tilde{h}\equiv I_r$. Consequently, $(\tilde{e}_1,\ldots,\tilde{e}_r)$ are linearly independent and unitary.
\end{proof}

\begin{lemma} \label{Lemma:NormIneq}
Let $v,w_1,\cdots,w_m$ be elements in a normed linear space $V$. For any $p>0$ and $\varepsilon>0$, there exists a constant $C$ depending only on $m,p,\varepsilon$ such that
\begin{equation*}
	\|v+w_1+\cdots+w_m\|^p \leqslant (1+\varepsilon)\|v\|^p + C(\|w_1\|^p+\cdots+\|w_m\|^p).
\end{equation*}
\end{lemma}

\begin{proof}
Case 1: $0<p\leqslant1$. Recall that, if $0<p\leqslant1$, then $(\sum_{j=0}^m t_j)^p \leqslant \sum_{j=0}^m t_j^p$ for any positive reals $t_j$. Therefore,
\begin{align*}
	\|v+w_1+\cdots+w_m\|^p & \leqslant (\|v\|+\|w_1\|+\cdots+\|w_m\|)^p \\
	& \leqslant \|v\|^p+\|w_1\|^p+\cdots+\|w_m\|^p.
\end{align*}
Case 2: $p>1$. By elementary calculus, it is easy to find a constant $C_{p,\varepsilon}>1$ such that
$$ (1+x)^p \leqslant (1+\varepsilon)+C_{p,\varepsilon} x^p,\quad x>0. $$
We set $w=w_1+\cdots+w_m$, then
$$ \|v+w\|^p \leqslant (\|v\|+\|w\|)^p \leqslant (1+\varepsilon)\|v\|^p + C_{p,\varepsilon}\|w\|^p. $$
Recall that, if $p>1$, then $(\sum_{j=1}^m t_j)^p \leqslant m^{p-1}(\sum_{j=1}^m t_j^p)$ for any positive reals $t_j$. Therefore,
\begin{gather*}
	\|w\|^p \leqslant (\|w_1\|+\cdots+\|w_m\|)^p \leqslant m^{p-1}(\|w_1\|^p+\cdots+\|w_m\|^p), \\
	\|v+w_1+\cdots+w_m\|^p \leqslant (1+\varepsilon)\|v\|^p + C_{p,\varepsilon} m^{p-1}(\|w_1\|^p+\cdots+\|w_m\|^p). \qedhere
\end{gather*}
\end{proof}

\section{A Quantitative Characterization of Griffiths Positivity} \label{Sec:Quan}

In the section, we prove a quantitative characterization of Griffiths positivity in terms of certain $L^2$-extension condition.

\begin{theorem} \label{Thm:QuanCharGrif}
Let $(E,h)$ be a holomorphic vector bundle over a domain $\Omega\subset\mathbb{C}^n$, equipped with a smooth Hermitian metric. Then the following are equivalent:
\begin{itemize}
\item[(i)] $i\Theta_h \geqslant_\textup{Grif} c\omega\otimes\textup{Id}_E$ at $x\in\Omega$, where $c\in\mathbb{R}$;
\item[(ii)] for any $\varepsilon>0$, there is a constant $\delta>0$ such that for any $\xi\in E_x$ and any holomorphic cylinder $x+P\subset\Omega$ with $\mathfrak{d}(P)<\delta$, there exists a holomorphic section $f\in\Gamma(x+P,E)$ satisfying $f(x)=\xi$ and
\begin{equation}
	\frac{1}{\textup{Vol}(P)}\int_{x+P}|f|_h^2d\lambda \leqslant (1-(c-\varepsilon)\mathfrak{d}(P)^2)|\xi|_h^2.
\end{equation}
\end{itemize}
\end{theorem}

\begin{remark} \label{Rmk:Grif}
The proof of the above theorem (with $c=0$) was written down by the second author during the preparation of \cite{XuZhou22}. Indeed, the same argument of (i) $\Rightarrow$ (ii) has been used in \cite[Theorem 5.9]{XuZhou22}. As one can observe, the proof for general $c\in\mathbb{R}$ is verbatim. Inayama's article \cite[v1]{Inayama} was submitted to arXiv on 16 Oct 2022. His result corresponds to the case of $n=1$ and $c\geqslant0$. As a comment on his article, the second author sent Inayama an email containing Theorem \ref{Thm:QuanCharGrif} on 20 Oct 2022. He recognized our proof, and the first version of the present article was subsequently submitted to arXiv a few days later. Inayama's proof for (ii) $\Rightarrow$ (i) combined some ideas from \cite[Theorem 6.4]{DWZZ18} and \cite[Theorem 1.6]{DNW21} with delicate calculations. However, our proof is a slight modification of \cite[Theorem 1.3]{DNWZ22}, and it appears to be simpler. Of course, the ideas behind \cite{DNW21,DNWZ22} eventually go back to Guan-Zhou's \cite{GuanZhou15} approach to the log-psh variation of relative Bergman kernels.
\end{remark}

\begin{proof}
\textit{(i) implies (ii).} We choose a holomorphic frame $(e_\alpha)_{\alpha=1}^r$ of $E$ in a neighborhood of $x$ such that $h_{\alpha\overline{\beta}}(x) = \delta_{\alpha\beta}$ and $dh_{\alpha\overline{\beta}}(x)=0$, where $h_{\alpha\overline{\beta}} := \langle{e_\alpha,e_\beta}\rangle_h$. Then $i\Theta_h \geqslant_\textup{Grif} c\omega\otimes \textup{Id}_E$ at $x$ means that
\begin{equation*}
	\sum_{i,j,\alpha,\beta} R_{i\overline{j}\alpha\overline{\beta}}(x) a_i\overline{a_j} \xi_\alpha\overline{\xi_\beta} = 
	\sum_{i,j,\alpha,\beta} -\frac{\partial^2h_{\alpha\overline{\beta}}}{\partial z_i\partial\overline{z}_j}(x) a_i\overline{a_j} \xi_\alpha\overline{\xi_\beta} \geqslant c|a|^2|\xi|^2
\end{equation*}
for any $a=(a_1,\ldots,a_n)\in\mathbb{C}^n$ and $\xi=(\xi_1,\cdots,\xi_r)\in\mathbb{C}^r$. By continuity, for every $\varepsilon>0$, we can find a sufficiently small neighborhood $U\subset\Omega$ of $x$ so that
\begin{equation*}
	\sum_{i,j,\alpha,\beta} -\frac{\partial^2h_{\alpha\overline{\beta}}}{\partial z_i\partial\overline{z}_j}(z) a_i\overline{a_j}\xi_\alpha\overline{\xi_\beta} \geqslant (c-\varepsilon)|a|^2|\xi|^2
\end{equation*}
for any $z\in U$, $a\in\mathbb{C}^n$ and $\xi\in\mathbb{C}^r$. Clearly, we can choose a constant $\delta>0$ such that $x+P\subset U$ for any holomorphic cylinder $P$ with $\mathfrak{d}(P)<\delta$.

Given a $\xi=\sum_\alpha\xi_\alpha e_\alpha(x)\in E_x$ and a holomorphic cylinder $P$ with $\mathfrak{d}(P)<\delta$, we define an $f\in\Gamma(x+P,E)$ by $f(z)=\sum_\alpha \xi_\alpha e_\alpha(z)$. Then $f(x)=\xi$ and
\begin{equation*}
	\frac{\partial^2}{\partial z_i\partial\overline{z}_j} \left((\varepsilon-c)|\xi|_h^2\cdot|z-x|^2-|f|_h^2\right) = \sum_{\alpha,\beta} \left((\varepsilon-c)\delta_{ij}\delta_{\alpha\beta} -\frac{\partial^2 h_{\alpha\overline{\beta}}}{\partial z_i\partial\overline{z}_j}\right)\xi_\alpha\overline{\xi_\beta}.
\end{equation*}
By the choice of $U$, $i\partial\bar{\partial}((\varepsilon-c)|\xi|_h^2\cdot|z-x|^2-|f|_h^2)\geqslant 0$ on $x+P$, then
$$ (\varepsilon-c)|\xi|_h^2\cdot|z-x|^2-|f|_h^2 $$
is a psh function on $x+P$. Therefore,
\begin{equation*}
	-|f(x)|_h^2 \leqslant \frac{1}{\textup{Vol}(P)}\int_{x+P} \left((\varepsilon-c)|\xi|_h^2\cdot|z-x|^2-|f|_h^2\right) d\lambda_z.
\end{equation*}
Using equation (\ref{Eq:MeanInt}), the above inequality can be reformulated as
\begin{equation*}
	\frac{1}{\textup{Vol}(P)}\int_{x+P}|f|_h^2d\lambda \leqslant (1-(c-\varepsilon)\mathfrak{d}(P)^2)|\xi|_h^2.
\end{equation*}

\textit{(ii) implies (i).} It is sufficient to show that $i\Theta_{h^*} \leqslant_\textup{Grif} -c\omega\otimes\textup{Id}_{E^*}$ at $x$. In view of Lemma \ref{Lemma:QuanGrif}, for any local holomorphic section $u\in\Gamma(U,E^*)$ with $u_x:=u(x)\neq 0$, we need to check that $i\partial\bar{\partial}(\log|u|_{h^*}^2)\geqslant c\omega$ at $x$.

We choose a vector $\xi\in E_x$ with $|\xi|_h=1$ and $|u_x|_{h^*}=|u_x(\xi)|$. Given $\varepsilon>0$, let $\delta>0$ be the same as in (ii). For any holomorphic cylinder $x+P\subset U$ with $\mathfrak{d}(P)<\delta$, there exists an $f\in\Gamma(x+P,E)$ such that $f_x=\xi$ and
\begin{equation*}
\frac{1}{\textup{Vol}(P)}\int_{x+P}|f|_h^2d\lambda \leqslant 1-(c-\varepsilon)\mathfrak{d}(P)^2.
\end{equation*}
Notice that $z\mapsto u_z(f_z)$ is a holomorphic function on $x+P$. Whenever $u_z(f_z)\neq 0$, we have
\begin{equation*}
\log|u_z|_{h^*}^2 \geqslant \log|u_z(f_z)|^2 - \log|f_z|_h^2.
\end{equation*}
Since the zero set of $u_z(f_z)$ is a set of zero measure, we have
\begin{align*}
&\, \frac{1}{\textup{Vol}(P)} \int_{x+P} \left(\log|u_z|_{h^*}^2-(c-\varepsilon)|z-x|^2\right) d\lambda_z \\
\geqslant &\, \frac{1}{\textup{Vol}(P)} \int_{x+P} \log|u_z(f_z)|^2 d\lambda_z - \frac{1}{\textup{Vol}(P)} \int_{x+P} \log|f_z|_h^2 d\lambda_z \\
& \qquad\qquad - \frac{1}{\textup{Vol}(P)} \int_{x+P}(c-\varepsilon)|z-x|^2d\lambda_z \\
\geqslant &\, \log|u_x(\xi)|^2 - \log\left( \frac{1}{\textup{Vol}(P)}\int_{x+P}|f_z|_h^2d\lambda_z \right) -(c-\varepsilon)\mathfrak{d}(P)^2 \\
\geqslant &\, \log|u_x|_{h^*}^2 - \log\left(1-(c-\varepsilon)\mathfrak{d}(P)^2\right) - (c-\varepsilon)\mathfrak{d}(P)^2 \\
\geqslant &\, \log|u_x|_{h^*}^2,
\end{align*}
where the second inequality follows from (\ref{Eq:MeanValIneq}), the Jensen's inequality and (\ref{Eq:MeanInt}), the last inequality follows from the fact that $\log(1+t)\leqslant t$.

By Corollary \ref{Cor:QuanPsh}, we know that $i\partial\bar{\partial}(\log|u|_{h^*}^2)\geqslant (c-\varepsilon)\omega$ at $x$. Since $\varepsilon>0$ is arbitrary, we conclude that $i\partial\bar{\partial}(\log|u|_{h^*}^2)\geqslant c\omega$ at $x$.
\end{proof}

\begin{remark}
The condition (ii) with $c=0$, which is equivalent to the Griffiths semi-positivity at the given point, appears to be weaker than the optimal $L^2$-extension condition given in Definition \ref{Def:OptLpExt}. On the other hand, if $\dim\Omega=1$ or $\operatorname{rank}E=1$, Griffiths semi-positivity is equivalent to Nakano semi-positivity, and hence implies the optimal $L^2$-extension condition. In general cases, it is still unclear whether Griffiths semi-positivity is equivalent to the optimal $L^2$-extension condition.
\end{remark}

\begin{corollary} \label{Cor:QuanCharGrif}
Let $(E,h)$ be a holomorphic vector bundle over a domain $\Omega\subset\mathbb{C}^n$, equipped with a smooth Hermitian metric. If there is a constant $c\in\mathbb{R}$ and a neighborhood $U$ of $x\in\Omega$ such that for any $\xi\in E_x$ and any holomorphic cylinder $x+P\subset U$, there exists an $f\in\Gamma(x+P,E)$ satisfying $f(x)=\xi$ and
\begin{equation}
	\frac{1}{\textup{Vol}(P)}\int_{x+P}|f|_h^2d\lambda \leqslant (1-c\,\mathfrak{d}(P)^2)|\xi|_h^2,
\end{equation}
then $i\Theta_h \geqslant_\textup{Grif} c\omega\otimes\textup{Id}_E$ at $x$.
\end{corollary}

Corollary \ref{Cor:QuanCharGrif} with $c=0$ is precisely a local version of Theorem \ref{Thm:DNWZ}.(2).

Let us discuss the relations between our results and that of Inayama \cite{Inayama}. In our notations, Theorem 1.4 of \cite{Inayama} can be reformulated as follows:

\begin{theorem}[see \cite{Inayama}] \label{Thm:Inayama}
Let $(E,h)$ be a Hermitian holomorphic vector bundle over a domain $\Omega\subset\mathbb{C}$. Given $x\in\Omega$, assume that there exists a $\gamma_x\in(0,\textup{dist}(x,\partial\Omega))$ and a lower semi-continuous function $g_x:[0,\gamma_x]\to\mathbb{R}_{\geqslant0}$ such that for any $\xi\in E_x\setminus\{0\}$ and $r\in(0,\gamma_x)$, there exists an $f\in\Gamma(x+\mathbb{D}_r,E)$ satisfying $f(x)=\xi$ and
\begin{equation}
	\frac{1}{\textup{Vol}(\mathbb{D}_r)} \int_{x+\mathbb{D}_r}|f|_h^2d\lambda \leqslant e^{-g_x(r)r^2}|\xi|_h^2,
\end{equation}
then $i\Theta_h \geqslant_\textup{Grif} 2g_x(0)\omega\otimes\textup{Id}_E$ at $x$.
\end{theorem}

In \cite{Inayama}, the infimum of $\int_{x+\mathbb{D}_r}|f|_h^2d\lambda / (\pi r^2|\xi|_h^2)$ over all holomorphic section $f\in\Gamma(x+\mathbb{D}_r,E)$ satisfying $f(x)=\xi$ is denoted by $L_h(x,r,\xi)$ and called the \textit{$L^2$-extension index} of $h$.
Notice that, in dimension 1, holomorphic cylinders are Euclidean discs and $\mathfrak{d}(\mathbb{D}_r)^2=\frac{1}{2}r^2$ by our definition \eqref{Eq:Diam}.

Now we assume all the conditions of Theorem \ref{Thm:Inayama}. For any fixed $0<\varepsilon\ll1$, by the semi-continuity, there exists a $r_\varepsilon\in(0,\gamma_x)$ so that $g_x(\cdot) \geqslant \max\{g_x(0)-\varepsilon,0\}$ on $[0,r_\varepsilon]$.
Notice that $e^{-t}\leqslant1-(1-\varepsilon)t$ for all $0<t\ll1$. By assumptions, for any $\xi\in E_x\setminus\{0\}$ and $0<r\ll r_\varepsilon$, there is an $f\in\Gamma(x+\mathbb{D}_r,E)$ such that $f(x)=\xi$ and
\begin{align*}
	\frac{1}{\textup{Vol}(\mathbb{D}_r)} \int_{x+\mathbb{D}_r}|f|_h^2d\lambda & \leqslant e^{-g_x(r)r^2}|\xi|_h^2 \leqslant e^{-\max\{g_x(0)-\varepsilon,0\}r^2} |\xi|_h^2 \\
	& \leqslant \big(1-(1-\varepsilon)\max\{g_x(0)-\varepsilon,0\}r^2\big) |\xi|_h^2.
\end{align*}
By Corollary \ref{Cor:QuanCharGrif}, we know
\begin{equation*}
	i\Theta_h \geqslant_\textup{Grif} 2(1-\varepsilon)\max\{g_x(0)-\varepsilon,0\} \omega\otimes\textup{Id}_E \text{ at }x.
\end{equation*}
Let $\varepsilon\to0$, we conclude that $i\Theta_h \geqslant_\textup{Grif} 2g_x(0) \omega\otimes\textup{Id}_E$ at $x$.

Conversely, we assume all the conditions of Corollary \ref{Cor:QuanCharGrif}. For any $c\in\mathbb{R}_{\geqslant0}$ and $0<r\ll1$, we have
\begin{equation*}
	1-c\,\mathfrak{d}(\mathbb{D}_r)^2 = 1-\frac{c}{2}r^2 = e^{-g_c(r)r^2}, \quad \text{where } g_c(r):=\frac{-\log(1-\frac{c}{2}r^2)}{r^2}.
\end{equation*}
Since $\lim_{r\to0}g_c(r)=\frac{c}{2}$, Theorem \ref{Thm:Inayama} yields $i\Theta_h \geqslant_\textup{Grif} c\omega\otimes\textup{Id}_E$ at $x$.

Therefore, Theorem \ref{Thm:Inayama} is equivalent to Corollary \ref{Cor:QuanCharGrif} with $n=1$ and $c\geqslant0$. There is a similar interpretation for Corollary 4.1 of \cite{Inayama}, which corresponds to Theorem \ref{Thm:QuanCharGrif} with $n=1$ and $c\geqslant0$.

\section{Characterizations of Pluriharmonicity} \label{Sec:PH}

Recall that, plurisubharmonicity is equivalent to the optimal $L^2$-extension condition and strictly psh implies sharper estimates in $L^2$ extensions.
In this section, we prove that pluriharmonic functions can be characterized by the equality part of the optimal $L^p$-extension condition. This answers a conjecture of Inayama \cite{Inayama} affirmatively.

\begin{theorem} \label{Thm:PluriHarm}
Let $\varphi:\Omega\to\mathbb{R}$ be a measurable function on a domain $\Omega\subset\mathbb{C}^n$ and $p>0$ be a constant, then the following conditions are equivalent:\par
\begin{itemize}
\item[(i)] $\varphi$ is pluriharmonic on $\Omega$;
\item[(ii)] for any holomorphic cylinder $x+P\subset\Omega$,
\begin{equation} \label{CondPH}
	\inf\left\{ \int_{x+P}|f|^pe^{-\varphi}d\lambda: f\in\mathcal{O}(x+P),f(x)=1 \right\} = \textup{Vol}(P)e^{-\varphi(x)}.
\end{equation}
\item[(iii)] there exists a positive continuous function $\gamma\ll1$ on $\Omega$ such that \eqref{CondPH} holds for any holomorphic cylinder $x+P\Subset\Omega$ with $\mathfrak{d}(P)<\gamma(x)$.
\end{itemize}
\end{theorem}

\begin{proof}
(i) $\Rightarrow$ (ii): Since $\varphi$ is pluriharmonic, there exists an $u\in\mathcal{O}(x+P)$ such that $p\operatorname{Re} u=\varphi$, and then $|e^u|^p=e^{\varphi}$. Clearly, $f:=e^{u-u(x)}$ is a holomorphic function on $x+P$ satisfying $f(x)=1$ and
\begin{equation*}
	\int_{x+P}|f|^pe^{-\varphi}d\lambda = \int_{x+P}|e^{-u(x)}|^pd\lambda = \textup{Vol}(P)e^{-\varphi(x)}.
\end{equation*}On the other hand, since $-\varphi$ is psh, for any $g\in\mathcal{O}(x+P)$, we know $|g|^pe^{-\varphi} = \exp(-\varphi+p\log|g|)$ is a psh function, then the mean value inequality says that
$$ \int_{x+P}|g|^pe^{-\varphi}d\lambda \geqslant \textup{Vol}(P)|g(x)|^pe^{-\varphi(x)}. $$
Therefore, \eqref{CondPH} holds for any holomorphic cylinder $x+P\subset\Omega$.

(ii) $\Rightarrow$ (iii): trivial.\\

(iii) $\Rightarrow$ (i): The proof is divided into two steps.

\textit{Step 1: $\varphi$ is lower semi-continuous and $-\varphi$ is psh.}

Given $x\in\Omega$, we choose a sequence $\{x_j\}_{j=1}^\infty$ such that $x_j\to x$ and
$$ \varliminf_{z\to x}\varphi(z) = \lim_{j\to+\infty} \varphi(x_j). $$
We choose a holomorphic cylinder $x+P\Subset\Omega$ with $\mathfrak{d}(P)<\gamma(x)$.

Let $\varepsilon>0$ and $0<s<1$ be fixed for the moment. By the condition \eqref{CondPH}, there exists an $f\in\mathcal{O}(x+P)$ such that $f(x)=1$ and
$$ \int_{x+P} |f|^pe^{-\varphi}d\lambda \leqslant \textup{Vol}(P)e^{-\varphi(x)+\varepsilon} < +\infty. $$
Since $\mathfrak{d}(sP)<\gamma(x_j)$ and $x_j+sP \Subset x+P$ for all $j\gg1$,
\begin{align*}
	s^{2n}\textup{Vol}(P)|f(x_j)|^pe^{-\varphi(x_j)} & \leqslant \int_{x_j+sP}|f|^pe^{-\varphi}d\lambda \\
	& \leqslant \int_{x+P}|f|^pe^{-\varphi}d\lambda \leqslant \textup{Vol}(P)e^{-\varphi(x)+\varepsilon}.
\end{align*}
Since $\lim_{j\to+\infty}f(x_j)=1$, letting $j\to+\infty$, we obtain
\begin{equation*}
	s^{2n} e^{-\lim_{j\to+\infty}\varphi(x_j)} \leqslant e^{-\varphi(x)+\varepsilon}.
\end{equation*}
Let $\varepsilon\searrow0$ and $s\nearrow1$, we conclude that
\begin{equation*}
	\varliminf_{z\to x}\varphi(z) = \lim_{j\to+\infty}\varphi(x_j) \geqslant \varphi(x).
\end{equation*}
Therefore, $\varphi$ is lower semi-continuous. Following the same idea of Inayama \cite{InayamaNote}, we can show that $-\varphi$ is psh:

Given a local holomorphic function $u$ on $U\subset\Omega$, the condition \eqref{CondPH} yields that $\int_{x+P}|u|^pe^{-\varphi}d\lambda \geqslant \textup{Vol}(P)|u(x)|^pe^{-\varphi(x)}$ for any holomorphic cylinder $x+P\Subset U$ with $\mathfrak{d}(P)<\gamma(x)$. Since $-\varphi$ is upper semi-continuous, it follows that $|u|^pe^{-\varphi}$ is a psh function on $U$. By Proposition \ref{Prop:EquivGrif}, the singular Hermitian metric $h:=e^{-\varphi}$ on the trivial line bundle is Griffiths semi-negative. Consequently, $-\varphi$ is psh.\\

\textit{Step 2: $\varphi$ is upper semi-continuous and psh.}

Given $x\in\Omega$, we choose a sequence $\{x_j\}_{j=1}^\infty$ such that $x_j\to x$ and
$$ \varlimsup_{z\to x}\varphi(z) = \lim_{j\to+\infty} \varphi(x_j). $$
Since $\varphi$ is lower semi-continuous, there is a constant $C$ such that $\varphi(x_j)\geqslant-C$ for all $j$. We choose a holomorphic cylinder $x+P\Subset\Omega$ with $\mathfrak{d}(P)<\gamma(x)$.

Let $\varepsilon>0$ and $0<s<1$ be fixed for the moment. Clearly, $\mathfrak{d}(P)<\gamma(x_j)$, $x_j+P\Subset\Omega$ and  $x+sP\Subset x_j+P$ for all $j\gg1$. By the condition \eqref{CondPH}, for each $j\gg1$, there exists an $f_j\in\mathcal{O}(x_j+P)$ such that $f_j(x_j)=1$ and
$$ \int_{x_j+P}|f_j|^pe^{-\varphi}d\lambda \leqslant \textup{Vol}(P)e^{-\varphi(x_j)+\varepsilon} < +\infty. $$
Since $\int_{x+sP}|f_j|^pe^{-\varphi}d\lambda \leqslant \textup{Vol}(P)e^{C+\varepsilon}$, by Lemma \ref{Lemma:GZ}, $\{f_j\}$ is uniformly bounded on any compact subset of $x+sP$. By Montel's theorem, there exists a subsequence $f_{j_k}|_{x+sP}$ that converges uniformly on any compact subset of $x+sP$ to some $f\in\mathcal{O}(x+sP)$. By Fatou's lemma,
\begin{align*}
	\int_{x+sP}|f|^pe^{-\varphi}d\lambda & = \int_{x+sP}\lim_{k\to+\infty}|f_{j_k}|^pe^{-\varphi}d\lambda \\
	& \leqslant \varliminf_{k\to+\infty} \int_{x+sP}|f_{j_k}|^pe^{-\varphi}d\lambda \leqslant \varliminf_{k\to+\infty} \textup{Vol}(P)e^{-\varphi(x_{j_k})+\varepsilon}.
\end{align*}
Since $f_{j_k}$ converges compactly to $f$ and $x_j\to x$, it is clear that
$$ f(x) = \lim_{k\to+\infty}f_{j_k}(x_{j_k}) = 1. $$
By the condition \eqref{CondPH},
\begin{equation*}
	s^{2n}\textup{Vol}(P)e^{-\varphi(x)} \leqslant \int_{x+sP}|f|^pe^{-\varphi}d\lambda \leqslant \textup{Vol}(P) \varliminf_{k\to+\infty}e^{-\varphi(x_{j_k})+\varepsilon},
\end{equation*}
i.e.
$$ s^{2n}e^{-\varphi(x)} \leqslant e^{-\varlimsup_{j\to+\infty}\varphi(x_{j_k})+\varepsilon}. $$
Let $\varepsilon\searrow0$ and $s\nearrow1$, we conclude that
$$ \varliminf_{z\to x}\varphi(z) = \lim_{k\to+\infty}\varphi(x_{j_k}) \leqslant \varphi(x). $$
Therefore, $\varphi$ is also upper semi-continuous.

Since $\varphi$ is upper semi-continuous, by \eqref{CondPH} and Montel's theorem, for any holomorphic cylinder $x+P\Subset\Omega$ with $\mathfrak{d}(P)<\gamma(x)$, there exists an $f\in\mathcal{O}(x+P)$ such that $f(x)=1$ and $\int_{x+P}|f|^pe^{-\varphi}d\lambda = \textup{Vol}(P)e^{-\varphi(x)}$. Therefore, $\varphi$ satisfies the optimal $L^p$-extension condition. By Theorem \ref{Thm:DNWZ}, $\varphi$ is also a psh function.

Since $\pm\varphi$ are psh, it follows that $\varphi$ is smooth and pluriharmonic.
\end{proof}

\begin{remark} \label{Rmk:NoInf}
In Theorem \ref{Thm:PluriHarm}, the assumption that $\varphi$ is $\mathbb{R}$-valued is necessary.

(1) We consider $p=2$ and $\varphi:=\log|z|^2$ on $\mathbb{C}$. Since $\varphi$ is harmonic on $\mathbb{C}\setminus\{0\}$, \eqref{CondPH} holds for any disc not containing $0$. Assume that $x+\mathbb{D}_r$ is a disc containing $0$ and $x\neq0$. Since $\varphi$ is psh, by the optimal $L^2$ extension theorem, there exists an $f\in\mathcal{O}(x+\mathbb{D}_r)$ such that $f(x)=1$ and $\int_{x+\mathbb{D}_r}|f|^2e^{-\varphi}d\lambda \leqslant \pi r^2e^{-\varphi(x)}$.
On the other hand, let $g\in\mathcal{O}(x+\mathbb{D}_r)$ be any holomorphic function satisfying $g(x)=1$ and $\int_{x+\mathbb{D}_r}|g|^2e^{-\varphi}d\lambda <+\infty$. Since $e^{-\varphi}$ is not integrable at $0$, we know $g(0)=0$. Then $g(z)/z$ is holomorphic function on $x+\mathbb{D}_r$ and
\begin{equation*}
	\int_{x+\mathbb{D}_r}|g|^2e^{-\varphi}d\lambda = \int_{x+\mathbb{D}_r}\left|\tfrac{g(z)}{z}\right|^2d\lambda_z \geqslant \pi r^2 \left|\tfrac{g(x)}{x}\right|^2 = \pi r^2 e^{-\varphi(x)}.
\end{equation*}
Therefore, \eqref{CondPH} also holds for $x+\mathbb{D}_r$. Finally, since $e^{-\varphi}$ is not integrable at $0$, for any $0+\mathbb{D}_r$, both sides of \eqref{CondPH} are $+\infty$.

In summary, $\varphi=\log|z|^2$ on $\mathbb{C}$ satisfies the condition (ii) with $p=2$. However, $\varphi(0)=-\infty$ and $\varphi$ is not harmonic at $0$.

(2) Assume that $\varphi:\Omega\to\mathbb{R}\cup\{+\infty\}$ is a measurable function satisfying the condition (iii).
If $\varphi\neq+\infty$ everywhere, then $\varphi$ is pluriharmonic. We claim that \textit{if $\varphi(x_0)=+\infty$ for some $x_0\in\Omega$, then $\varphi$ is identically $+\infty$ on $\Omega$.}

Since $\varphi\neq-\infty$, the same argument of Step 1 shows that $-\varphi$ is a psh function. Suppose $\varphi(x)=+\infty$ and we choose a holomorphic cylinder $x+P\Subset\Omega$ with $\mathfrak{d}(P)<\gamma(x)$.
Since $e^{-\varphi(x)}=0$, the equation \eqref{CondPH} means that there exists a sequence $\{f_j\}_{j=1}^\infty$ of holomorphic functions on $x+P$ such that $f_j(x)=1$ and $\int_{x+P}|f_j|^pe^{-\varphi}d\lambda \leqslant j^{-1}$.
If $\varphi\not\equiv+\infty$ on $x+P$, then it follows from Lemma \ref{Lemma:GZ} that $\{f_j\}_{j=1}^\infty$ is uniformly bounded on any compact subset of $x+P$.
By Montel's theorem, there exists a subsequence $\{f_{j_k}\}_{k=1}^\infty$ that converges compactly to some $f\in\mathcal{O}(x+P)$. Clearly, $f(x)=1$. By Fatou's Lemma,
\begin{equation*}
	\int_{x+P}|f|^pe^{-\varphi}d\lambda = \int_{x+P}\lim_{k\to+\infty}|f_{j_k}|^pe^{-\varphi}d\lambda \leqslant \varliminf_{k\to+\infty}\int_{x+P}|f_{j_k}|^pe^{-\varphi}d\lambda = 0.
\end{equation*}
Since $f\not\equiv0$, it follows that $\varphi=+\infty$ almost everywhere on $x+P$. Since $-\varphi$ is a psh function, it turns out that $\varphi\equiv+\infty$ on $x+P$. A contradiction!

Therefore, if $\varphi(x)=+\infty$, then $\varphi\equiv+\infty$ on any $x+P$ with $\mathfrak{d}(P)<\gamma(x)$. By the continuity of $\gamma>0$, it is easy to show that, if $\varphi(x_0)=+\infty$ for some $x_0\in\Omega$, then $\varphi$ is identically $+\infty$ on $\Omega$.
\end{remark}

In the case of $n=1$ and $p=2$, assuming the subharmonicity in advance, one can prove a stronger result: the equality in a single disc guarantees the harmonicity in that disc. This result is a consequence of Theorem 1.11 of Guan-Mi \cite{GuanMi22}. Since the origin proof is lengthy, we give a shorter proof by using Corollary 1.5 of \cite{GuanMi22}.

\begin{theorem}
Let $\varphi>-\infty$ be a subharmonic function on $\mathbb{D}$, then $\varphi$ is harmonic on $\mathbb{D}$ if and only if
$$ B_{\mathbb{D}}(0;e^{-\varphi}) = \pi^{-1}e^{\varphi(0)}. $$
\end{theorem}

\begin{proof}
Clearly, we only need to prove the sufficiency.
Suppose, on the contrary, $T:=i\partial\bar{\partial}\varphi\not\equiv0$ on $\mathbb{D}$. Then we can choose a small disc $\mathbb{D}(x;r)\Subset\mathbb{D}$ such that $0\notin\mathbb{D}(x;r)$ and $T\not\equiv0$ on $\mathbb{D}(x;\frac{r}{3})$.
(Otherwise, $\operatorname{supp} T\subset\{0\}$ and $T=c\delta_0$ for some $c\geqslant0$. Since $\varphi(0)>-\infty$, it turns out that $c=0$ and $T\equiv0$.)

Let $\eta$ be a smooth function such that $0\leqslant\eta\leqslant1$, $\operatorname{supp}\eta\subset\mathbb{D}(x;\frac{2r}{3})$ and $\eta\equiv1$ on $\mathbb{D}(x;\frac{r}{3})$. Then $\eta T\not\equiv0$ is a positive $(1,1)$-current on $\mathbb{D}$. Since we are working in dimension $1$, there exists a subharmonic function $u$ on $\mathbb{D}$ such that $i\partial\bar{\partial} u=\eta T$. Since $\operatorname{supp}\eta T\subset\mathbb{D}(x;\frac{2r}{3})$, we see that $u$ is harmonic on $\mathbb{D}\setminus\overline{\mathbb{D}(x;\frac{2r}{3})}$.

Solving the Dirichlet problem, we find a harmonic function $v$ on $\mathbb{D}(x;r)$ such that $v$ is continuous up to the boundary and $u=v$ on $\partial\mathbb{D}(x;r)$. By the maximum principle, $u\leqslant v$ on $\mathbb{D}(x;r)$. Since $\eta T\not\equiv0$ on $\mathbb{D}(x;r)$, we know $u\not\equiv v$. We define
\begin{equation*}
	\sigma(z) := \begin{cases} v(z) - u(z), & z\in\overline{\mathbb{D}(x;r)} \\ 0, & z\in\mathbb{D}\setminus\mathbb{D}(x;r) \end{cases}.
\end{equation*}
Since $v$ is harmonic, $i\partial\bar{\partial}\sigma = -i\partial\bar{\partial} u = -\eta T$ on $\mathbb{D}(x;r)$. Since $v-u$ is a nonnegative harmonic function on $\mathbb{D}(x;r)\setminus\overline{\mathbb{D}(x;\frac{2r}{3})}$ and $v-u=0$ on $\partial\mathbb{D}(x;r)$, it is easy to see that $\sigma$ is subharmonic on $\mathbb{D}\setminus\overline{\mathbb{D}(x;\frac{2r}{3})}$.

We set $\widetilde{\varphi}:=\varphi+\sigma$, then $i\partial\bar{\partial}\widetilde{\varphi}=(1-\eta)T\geqslant0$ on $\mathbb{D}(x;r)$ and $i\partial\bar{\partial}\widetilde{\varphi}\geqslant i\partial\bar{\partial}\varphi\geqslant0$ on $\mathbb{D}\setminus\overline{\mathbb{D}(x;\frac{2r}{3})}$.
After a modification on a set of zero measure, $\widetilde{\varphi}$ is also a subharmonic function on $\mathbb{D}$. By the construction, $\widetilde{\varphi}\geqslant\varphi$ on $\mathbb{D}$, $\widetilde{\varphi}\equiv\varphi$ on $\mathbb{D}\setminus\mathbb{D}(x;r)$ and $\widetilde{\varphi}\not\equiv\varphi$ on $\mathbb{D}(x;r)$.\\

\textit{Following the essential idea of \cite[Corollary 1.5]{GuanMi22}, the existence of such $\widetilde{\varphi}$ will lead to a contradiction.} For the convenience of readers, we recall the details.

We set $\psi:=2\log|z|$, then $e^{-\varphi-\psi}$ is locally integrable on $\mathbb{D}\setminus\{0\}$ and not integrable at $0$. For any $s\in(0,1]$, we define $U_s:=\{\psi<\log s\}=\mathbb{D}(0;\sqrt{s})$ and
\begin{equation*}
	I_\varphi(s) := \inf\left\{ \int_{U_s}|f|^2e^{-\varphi}d\lambda: f\in A^2(U_s;e^{-\varphi}), f(0)=1 \right\} = \frac{1}{B_{U_s}(0;e^{-\varphi})}.
\end{equation*}
We denote by $F_s\in A^2(U_s;e^{-\varphi})$ the unique holomorphic function such that $F_s(0)=1$ and $\int_{U_s}|F_s|^2e^{-\varphi}d\lambda = I_\varphi(s)$. We define $I_{\widetilde{\varphi}}(s)$ and $\tilde{F}_s$ in similar ways.

By the concavity of minimal $L^2$ integrals (see \cite{Guan19}), $I_\varphi(s)$ and $I_{\widetilde{\varphi}}(s)$ are concave increasing functions on $(0,1]$. By the optimal $L^2$ extension theorem (see \cite{GuanZhou15}),
$$ I_\varphi(s) \leqslant \pi e^{-\varphi(0)}s, \quad 0<s\leqslant1. $$
By assumption, $I_\varphi(1) = \pi e^{-\varphi(0)}$. Since $\lim_{s\to0}I_\varphi(s)=0$, the concavity implies that $I_\varphi(s) \equiv \pi e^{-\varphi(0)}s$. Therefore, $I_\varphi(s)$ is a linear function of $s\in(0,1]$, and then \cite[Remark 5.3]{XuZhou22} says that $F_s\equiv F_1|_{U_s}$ for all $0<s<1$.

Since $\varphi>-\infty$, we know $\mathcal{I}(\varphi)=\mathcal{O}_\mathbb{D}$. Since $\int_\mathbb{D} |\tilde{F}_1|^2e^{-\widetilde{\varphi}} d\lambda<+\infty$ and $\widetilde{\varphi}\equiv\varphi$ on $\mathbb{D}\setminus\mathbb{D}(x;r)$, it is clear that $\int_\mathbb{D} |\tilde{F}_1|^2e^{-\varphi} d\lambda<+\infty$, i.e. $\tilde{F}_1\in A^2(\mathbb{D};e^{-\varphi})$.
We choose $a\in(0,1)$ so that $\sqrt{a}>|x|+r$, then $\widetilde{\varphi} \equiv \varphi$ on $\mathbb{D}\setminus U_a$. Since $\int_{U_a} |\tilde{F}_1|^2 e^{-\widetilde{\varphi}} d\lambda \geqslant I_{\widetilde{\varphi}}(a)$, it is clear that
\begin{equation*}
	I_{\widetilde{\varphi}}(1) - I_{\widetilde{\varphi}}(a) \geqslant \int_{\mathbb{D}\setminus U_a} |\tilde{F}_1|^2 e^{-\widetilde{\varphi}} d\lambda = \int_{\mathbb{D}\setminus U_a} |\tilde{F}_1|^2 e^{-\varphi} d\lambda.
\end{equation*}
By the minimality of $F_s\in A^2(U_s;e^{-\varphi})$, we know
\begin{equation*}
	\int_{U_s} |\tilde{F}_1|^2 e^{-\varphi} d\lambda = \int_{U_s} |F_s|^2 e^{-\varphi} d\lambda + \int_{U_s} |\tilde{F}_1-F_s|^2 e^{-\varphi} d\lambda.
\end{equation*}
Since $F_s\equiv F_1|_{U_s}$, it follows that
\begin{align*}
	I_{\widetilde{\varphi}}(1) - I_{\widetilde{\varphi}}(a) & \geqslant \int_\mathbb{D} |F_1|^2 e^{-\varphi} d\lambda - \int_{U_a} |F_1|^2 e^{-\varphi} d\lambda + \int_{\mathbb{D}\setminus U_a} |\tilde{F}_1-F_1|^2 e^{-\varphi} d\lambda \\
	& \geqslant I_\varphi(1) - I_\varphi(a).
\end{align*}

We choose $b\in(0,1)$ such that $\sqrt{b}<|x|-r$. Since $\widetilde{\varphi}\equiv\varphi$ on $U_b$, we know $I_{\widetilde{\varphi}}(b) = I_\varphi(b)$.
Since $\widetilde{\varphi}\geqslant\varphi$ and  $\widetilde{\varphi} \not\equiv \varphi$ on $\mathbb{D}(x;r)\subset U_a$,
it is clear that $$U_a\cap\{\widetilde{\varphi}>\varphi\}\cap\{F_1\neq0\}$$ is a set of positive measure. Consequently,
\begin{equation*}
	I_{\widetilde{\varphi}}(a) \leqslant \int_{U_a} |F_1|^2e^{-\widetilde{\varphi}} d\lambda < \int_{U_a} |F_1|^2e^{-\varphi} d\lambda = I_\varphi(a).
\end{equation*}
Recall that, $I_\varphi(s)=\pi e^{-\varphi(0)}s$ for all $0<s\leqslant1$. Then
\begin{align*}
	\frac{I_{\widetilde{\varphi}}(a) - I_{\widetilde{\varphi}}(b)}{a-b} < \frac{I_{\varphi}(a) - I_{\varphi}(b)}{a-b} & = \pi e^{-\varphi(0)} \\ & = \frac{I_{\varphi}(1) - I_{\varphi}(a)}{1-a} \leqslant \frac{I_{\widetilde{\varphi}}(1) - I_{\widetilde{\varphi}}(a)}{1-a}.
\end{align*}
However, this contradicts with the concavity of $s\mapsto I_{\widetilde{\varphi}}(s)$.

Therefore, $i\partial\bar{\partial}\varphi\equiv0$, and $\varphi$ is a harmonic function on $\mathbb{D}$.
\end{proof}

\section{A Characterization of Flatness} \label{Sec:Flat}

In this section, we prove that a singular Hermitian metric is smooth and flat if and only if it satisfies the equality part of the optimal $L^p$-extension condition.

\begin{theorem} \label{Thm:Flat}
Let $E$ be a holomorphic vector bundle over a domain $\Omega\subset\mathbb{C}^n$ and $p>0$ be a constant. Let $h$ be a singular Hermitian metric on $E$ such that $0<\det h<+\infty$ everywhere, then the following conditions are equivalent:\par
\begin{itemize}
\item[(i)] $h$ is smooth and $\Theta_h\equiv0$;
\item[(ii)] for any holomorphic cylinder $x+P\subset\Omega$ and any $v\in E_x$,
\begin{equation}\label{CondFlat}
	\inf\left\{ \int_{x+P}|f|_h^pd\lambda: f\in\Gamma(x+P,E), f(x)=v \right\} = \textup{Vol}(P)|v|_h^p.
\end{equation}
\item[(iii)] there exists a positive continuous function $\gamma\ll1$ on $\Omega$ such that \eqref{CondFlat} holds for any holomorphic cylinder $x+P\Subset\Omega$ with $\mathfrak{d}(P)<\gamma(x)$ and any $v\in E_x$.
\end{itemize}
\end{theorem}

\begin{proof}
(i) $\Rightarrow$ (ii): Since $\Theta_h\equiv0$, by Lemma \ref{Lemma:FlatUnitary}, there exists a unitary holomorphic frame $(e_1,\cdots,e_r)$ of $(E,h)$ on $x+P$.
Assume that $v=\sum_jc_je_j(x)$, then we define a holomorphic section $f\in\Gamma(x+P,E)$ by $f=\sum c_je_j$. Clearly,
$$ \int_{x+P}|f|_h^pd\lambda = \int_{x+P} \big(\sum\nolimits_j|c_j|^2\big)^{p/2} d\lambda = \textup{Vol}(P)|v|_h^p. $$
On the other hand, since $(E,h)$ is Griffiths semi-negative, for any $g\in\Gamma(x+P,E)$, we know $|g|_h^p$ is a psh function, then the mean-value inequality yields
$$ \int_{x+P}|g|_h^pd\lambda \geqslant \textup{Vol}(P)|g(x)|_h^p.$$
Therefore, \eqref{CondFlat} holds for any $x+P\Subset\Omega$ and $v\in E_x$.

(ii) $\Rightarrow$ (iii): trivial.\\

(iii) $\Rightarrow$ (i): The proof of this part is divided into four steps.

\textit{Step 1: $h$ is upper semi-continuous and Griffiths semi-negative.}

For any sequence $v_j\in E_{x_j}$ converging to $v\in E_x$, we need to show that
$$ \varlimsup_{j\to+\infty} |v_j|_h \leqslant |v|_h. $$
In the following, we fix a holomorphic cylinder $x+P\Subset\Omega$ with $\mathfrak{d}(P)<\gamma(x)$.

Let $\varepsilon,\delta>0$ and $0<s<1$ be given. We choose a basis $\{\xi_1,\cdots,\xi_r\}$ of $E_x$, then there exist holomorphic sections $f_\alpha\in\Gamma(x+P,E)$ such that $f_\alpha(x)=\xi_\alpha$ and
$$ \int_{x+P}|f_\alpha|_h^pd\lambda \leqslant (1+\varepsilon)\textup{Vol}(P)|\xi_\alpha|_h^p < +\infty. $$
By the continuity, $f_\alpha$ are linearly independent in some neighborhood $U$ of $x$. Clearly, $x_j\in U$, $x_j+sP\Subset x+P$ and $\mathfrak{d}(sP)<\gamma(x_j)$ for all $j\gg1$.

We write $v_j = \sum_\alpha c_{j,\alpha}f_\alpha(x_j)$ and $v = \sum_\alpha c_\alpha\xi_\alpha$. Since $v_j\to v$, it is clear that $c_{j,\alpha}\to c_\alpha$. Without loss of generality, we assume that $c_2=\cdots=c_r=0$, i.e. $v=c_1\xi_1$.
By Lemma \ref{Lemma:NormIneq}, there is a constant $C$ depending only on $r,p,\delta$ such that
\begin{align*}
	\left|\sum\nolimits_\alpha c_{j,\alpha} f_\alpha\right|_h^2 \leqslant (1+\delta)|c_1f_1|_h^p
	+ C \sum\nolimits_\alpha |(c_{j,\alpha}-c_\alpha)f_\alpha|_h^p.
\end{align*}
Since $x_j+sP\Subset x+P$, we have
\begin{align*}
	&\, s^{2n}\textup{Vol}(P) \left|\sum\nolimits_\alpha c_{j,\alpha} f_\alpha(x_j)\right|_h^p \leqslant \int_{x_j+sP} \left|\sum\nolimits_\alpha c_{j,\alpha} f_\alpha\right|_h^p d\lambda \\
	\leqslant &\, (1+\delta)|c_1|^p \int_{x_j+sP}|f_1|_h^pd\lambda
	+ C \sum\nolimits_\alpha |c_{j,\alpha}-c_\alpha|^p \int_{x_j+sP}|f_\alpha|_h^p d\lambda \\
	\leqslant &\, (1+\delta)|c_1|^p \int_{x+P}|f_1|_h^pd\lambda
	+ C \sum\nolimits_\alpha |c_{j,\alpha}-c_\alpha|^p \int_{x+P}|f_\alpha|_h^p d\lambda.
\end{align*}
Consequently,
\begin{align*}
	\varlimsup_{j\to+\infty} s^{2n}\textup{Vol}(P) \left|\sum\nolimits_\alpha c_{j,\alpha} f_\alpha(x_j)\right|_h^p & \leqslant (1+\delta)|c_1|^p \int_{x+P} |f_1|_h^p d\lambda \\
	& \leqslant (1+\delta)|c_1|^p(1+\varepsilon)\textup{Vol}(P)|\xi_1|_h^p.
\end{align*}
Let $s\nearrow1$, $\varepsilon\searrow0$ and $\delta\searrow0$, we get
\begin{gather*}
	\varlimsup_{j\to+\infty} |v_j|_h^p = \varlimsup_{j\to+\infty} \left|\sum\nolimits_\alpha c_{j,\alpha} f_\alpha(x_j)\right|_h^p \leqslant |c_1|^p|\xi_1|_h^p = |v|_h^p.
\end{gather*}

In conclusion, the metric $h:E\to[0,+\infty)$ is upper semi-continuous.

Given a holomorphic section $f\in\Gamma(U,E)$, for any holomorphic cylinder $x+P\Subset U$ with $\mathfrak{d}(P)<\gamma(x)$, the equation \eqref{CondFlat} yields
$$ \int_{x+P}|f|_h^pd\lambda \geqslant \textup{Vol}(P)|f(x)|_h^p. $$
Since $|f|_h$ is upper semi-continuous, it follows that $|f|_h^p$ is a psh function. By Proposition \ref{Prop:EquivGrif}, $h$ is Griffiths semi-negative.\\

\textit{Step 2: $h$ is lower semi-continuous and Griffiths semi-positive.}

For any sequence $v_j\in E_{x_j}$ converging to $v\in E_x$, we need to show that
$$ \varliminf_{j\to+\infty} |v_j|_h \geqslant |v|_h. $$
By passing to a subsequence, we may assume that the limit of $|v_j|_h$ exists. Since $h$ is upper semi-continuous, there is a constant $C$ so that $|v_j|_h\leqslant C$ for all $j$. In the following, we fix a holomorphic cylinder $x+P\Subset\Omega$ with $\mathfrak{d}(P)<\gamma(x)$.

Let $\varepsilon>0$ and $0<s<1$ be fixed for the moment. Clearly, $\mathfrak{d}(P)<\gamma(x_j)$, $x_j+P\Subset\Omega$ and  $x+sP\Subset x_j+P$ for all $j\gg1$. By \eqref{CondFlat}, for each $j\gg1$, there exists an $f_j\in\Gamma(x_j+P,E)$ such that $f_j(x_j)=v_j$ and
$$ \int_{x_j+P}|f_j|_h^pd\lambda \leqslant (1+\varepsilon)\textup{Vol}(P)|v_j|_h^p. $$

Let $U\subset x+sP$ be an open subset such that $E|_U$ is trivial and $K$ be a compact subset of $U$. We choose a holomorphic frame field $(e_1,\cdots,e_r)$ of $E|_U$, then $h$ can be regarded as a family of positive definite Hermitian matrices on $U$. Since $h$ is Griffiths semi-negative, by Lemma \ref{Lemma:PT}, there exists a constant $c_K>0$ such that $h\geqslant c_K(\det h) I_r$ on $K$. Moreover, $\varphi:=\log\det h>-\infty$ is a psh function on $U$. We write $f_j|_U=\sum_\alpha f_{j,\alpha}e_\alpha$, where $f_{j,\alpha}\in\mathcal{O}(U)$, then
$$ |f_j|_h^2 \geqslant c_K \sum\nolimits_{\alpha=1}^r|f_{j,\alpha}|^2e^{\varphi}. $$
Since $\sum_{j=1}^r t_j^{p/2} \leqslant \max\{1,r^{1-p/2}\} (\sum_{j=1}^r t_j)^{p/2}$ for any positive reals $t_j$, we have
\begin{equation*}
	\sum\nolimits_{\alpha=1}^r|f_{j,\alpha}|^pe^{p\varphi/2} \leqslant \max\{1,r^{1-p/2}\} c_K^{-p/2} |f_j|_h^p.
\end{equation*}
Therefore,
\begin{equation*}
	\int_K \sum\nolimits_\alpha |f_{j,\alpha}|^pe^{p\varphi/2} d\lambda \leqslant \max\{1,r^{1-p/2}\} c_K^{-p/2} (1+\varepsilon)\textup{Vol}(P)C^p < +\infty.
\end{equation*}
By similar arguments as Theorem \ref{Thm:PluriHarm}, there exists a constant $C_K$ such that
$$ \sum\nolimits_\alpha \int_K|f_{j,\alpha}|^{p/2}d\lambda \leqslant C_K < +\infty. $$
By Montel's theorem and diagonal argument, there exists a subsequence $f_{j_k}|_{x+sP}$ that converges uniformly on any compact subset of $x+sP$ to some $f\in\Gamma(x+sP,E)$. By Fatou's lemma,
\begin{align*}
	\int_{x+sP}|f|_h^pd\lambda & = \int_{x+sP}\lim_{k\to+\infty}|f_{j_k}|_h^pd\lambda \\
	& \leqslant \varliminf_{k\to+\infty} \int_{x+sP}|f_{j_k}|_h^pd\lambda \leqslant \lim_{k\to+\infty} 
	(1+\varepsilon)\textup{Vol}(P)|v_{j_k}|_h^p.
\end{align*}
Since $f_j$ converges compactly to $f$ and $x_j\to x$, it is clear that
$$ f(x) = \lim_{k\to+\infty}f_{j_k}(x_{j_k}) = \lim_{k\to+\infty}v_{j_k} = v. $$
By \eqref{CondFlat},
\begin{equation*}
	s^{2n}\textup{Vol}(P)|v|_h^p \leqslant \int_{x+sP}|f|_h^pd\lambda \leqslant (1+\varepsilon)\textup{Vol}(P) \lim_{k\to+\infty}|v_{j_k}|_h^p.
\end{equation*}
Let $\varepsilon\searrow0$ and $s\nearrow1$, we conclude that
$$ |v|_h \leqslant \lim_{k\to+\infty} |v_{j_k}|_h = \varliminf_{j\to+\infty} |v_j|_h. $$

In conclusion, the Hermitian metric $h:E\to[0,+\infty)$ is continuous.

For any holomorphic cylinder $x+P\Subset\Omega$ with $\mathfrak{d}(P)<\gamma(x)$ and any $v\in E_x$, by using Montel's theorem, we can find an $f\in\Gamma(x+P,E)$ such that $f(x)=v$ and $\int_{x+P}|f|_h^pd\lambda = \textup{Vol}(P)|v|_h^p$. Therefore, $(E,h)$ satisfies the optimal $L^p$-extension condition. By Theorem \ref{Thm:DNWZ}, $h$ is also Griffiths semi-positive.\\

\textit{Step 3: $\Theta_h:=\overline{\partial}(h^{-1}\partial h)=0$ in the sense of currents.}

In the following, let $U\Subset V\Subset\Omega$ be open subsets such that $E|_V$ is trivial. We fix a holomorphic frame $(e_1,\ldots,e_r)$ of $E|_V$, then $h$ can be regarded as a continuous family of positive definite Hermitian matrices on $V$.

According to Lemma \ref{Lemma:PT}, there exists a sequence of smooth Hermitian metrics $\{h_\nu\}_{\nu=1}^\infty$ with Griffiths negative curvature, decreasingly converging to $h$ pointwise on $\overline{U}$. Since
\begin{equation*}
	\langle{e_\alpha,e_\beta}\rangle_h = \frac{1}{4} \sum\nolimits_{k=0}^3 i^k |e_\alpha+i^ke_\beta|_h^2,
\end{equation*}
by Dini's theorem, $\langle{e_\alpha,e_\beta}\rangle_{h_\nu}$ converges to $\langle{e_\alpha,e_\beta}\rangle_h$ uniformly on $\overline{U}$, i.e. $h_\nu$ converges to $h$ uniformly on $\overline{U}$.

Since $h$ is continuous, according to Lemma \ref{Lemma:Raufi}, the entries of $\partial h$ are $L_\textup{loc}^2$-forms, the entries of $\Theta_h := \overline{\partial}(h^{-1}\partial h)$ are currents with measure coefficients, and $\Theta_{h_\nu} := \overline{\partial}(h_\nu^{-1}\partial h_\nu)$ converge weakly to $\Theta_h$ as currents with measure coefficients.

Since $h_\nu$ are Griffiths negative, for any $\xi\in C^0(U,\mathbb{C}^r)$ and any strongly positive test-form $\phi\in C_c^0(U,\wedge^{n-1,n-1}T_\Omega^*)$, we have
\begin{equation*}
	\int_U \langle{i\Theta_{h_\nu}\xi,\xi}\rangle_{h_\nu} \wedge \phi
	= \int_U i(\xi^*h_\nu\Theta_{h_\nu}\xi) \wedge \phi \leqslant 0.
\end{equation*}
Since $h_\nu$ converges uniformly to $h$ and $\Theta_{h_\nu}$ converges weakly to $\Theta_h$, we see that
$$ \int_U i(\xi^*h\Theta_h\xi) \wedge \phi \leqslant 0. $$

Notice that, the dual metric $g:=(h^{-1})^t$ of $h$ is also Griffiths semi-negative. By similar arguments, the entries of $\partial g$ are $L_\textup{loc}^2$-forms, the entries of $\Theta_g := \overline{\partial}(g^{-1}\partial g)$ are currents with measure coefficients, and
\begin{equation} \label{Eq:semipos}
	\int_U i(\eta^*g\Theta_g\eta) \wedge \phi \leqslant 0
\end{equation}
for any $\eta\in C^0(U,\mathbb{C}^r)$ and any strongly positive test-form $\phi\in C_c^0(U,\wedge^{n-1,n-1}T_\Omega^*)$.

Since the entries of $\partial h$ and $\partial(h^{-1})=(\partial g)^t$ are $L_\textup{loc}^2$-forms, it is clear that
$$ (\partial h) h^{-1} + h\,\partial(h^{-1}) = 0. $$
Consequently,
\begin{equation*}
	g^{-1}\partial g = h^t \partial (h^{-1})^t = \big( \partial(h^{-1}) h \big)^t = -\big( h^{-1} \partial h \big)^t, \quad
	\Theta_g = -\Theta_h^t.
\end{equation*}
Notice that,
\begin{equation*}
	\eta^* g\Theta_g \eta = (\eta^* g\Theta_g \eta)^t = (\eta^th^{-1}) h\Theta_h (h^{-1}\overline{\eta})
	= - \xi^* h\Theta_h \xi,
\end{equation*}
where $\xi:=h^{-1}\overline{\eta} \in C^0(U,\mathbb{C}^r)$. Then the inequality \eqref{Eq:semipos} can be reformulated as
$$ \int_U i(\xi^*h\Theta_h\xi) \wedge \phi \geqslant 0. $$

In summary, $\int_U i(\xi^*h\Theta_h\xi) \wedge \phi = 0$ for any continuous section $\xi\in C^0(U,\mathbb{C}^r)$ and any strongly positive test-form $\phi\in C_c^0(U,\wedge^{n-1,n-1}T_\Omega^*)$. By linear combinations, we conclude that
$$ \int_U (\eta^*\Theta_h\xi) \wedge\phi = 0 $$
for all $\eta,\xi\in C^0(U,\mathbb{C}^r)$ and $\phi\in C_c^0(U,\wedge^{n-1,n-1}T_\Omega^*)$, which implies that $\Theta_h\equiv0$ in the sense of currents.\\

\textit{Step 4: $h$ is smooth and flat.}

We fix a local holomorphic frame of $E$, then $h$ can be regarded as a continuous family of Hermitian matrices. Since $\Theta_h := \overline{\partial}(h^{-1}\partial h) = 0$ in the sense of currents, we see that the entries of $h^{-1}\partial h$ are  holomorphic $1$-forms. Taking conjugation, we see that $(h^{-1}\partial h)^* = (\overline{\partial}h)h^{-1}$ is anti-holomorphic.
In particular, $\partial_jh$ and $\overline{\partial}_kh$ are continuous, which means that $h$ is $C^1$-smooth.

Since $h^{-1}\partial h,(\overline{\partial}h)h^{-1}\in C^\infty$ and $h\in C^1$, we see that
\begin{equation*}
	\partial_j h = h(h^{-1}\partial_jh) \quad\text{and}\quad \overline{\partial}_kh = ((\overline{\partial}_kh)h^{-1})h
\end{equation*}
are $C^1$-smooth. Consequently, $h$ is $C^2$-smooth. Repeating this process, we conclude that $h$ is $C^\infty$-smooth. Since $h$ is both Griffiths semi-positive and semi-negative (in the usual sense), it is clear that $\Theta_h\equiv0$.
\end{proof}

\section{Appendix: The Weighted $p$-Bergman Kernel}

Let $\Omega\subset\mathbb{C}^n$ be an open set and $\varphi$ be a measurable function on $\Omega$ which is locally bounded from above. We always assume that $\{z\in\Omega: \varphi(z)=-\infty\}$ is a set of zero measure. For any $p>0$, the weighted $p$-Bergman space of $\Omega$ is defined as
\begin{equation*}
	A^p(\Omega;e^{-\varphi}) := \left\{ f\in\mathcal{O}(\Omega): \|f\|_p= \big( \int_\Omega|f|^pe^{-\varphi}d\lambda \big)^{1/p} < +\infty \right\}.
\end{equation*}
For any compact set $K\subset\Omega$, there is a constant $C>0$ such that $\sup_K|f| \leqslant C\|f\|_p$ for all $f\in A^p(\Omega;e^{-\varphi})$.
Consequently, $A^p(\Omega;e^{-\varphi})$ is a Banach space (resp. complete metric space) for $p\geqslant1$ (resp. $0<p<1$) and the evaluation maps $\mathbf{ev}_x:f\mapsto f(x)$ are continuous linear functionals on $A^p(\Omega;e^{-\varphi})$. 
Similar to the classical Bergman theory, the weighted $p$-Bergman kernel of $\Omega$ is defined as
\begin{equation*}
B_{\Omega,p}(x;e^{-\varphi}) := \|\textbf{ev}_x\|^p = \sup\{ |f(x)|^p: f\in A^p(\Omega;e^{-\varphi}), \|f\|_p\leqslant1 \}.
\end{equation*}
Using Montel's theorem, there exists an $f\in A^p(\Omega;e^{-\varphi})$ achieving the supremum.
When $p=2$, it is well-known that $B_{\Omega,2}(\cdot;e^{-\varphi})$ is real-analytic. For general $p>0$, Chen-Zhang \cite{ChenZhang22}  showed that $B_{\Omega,p}(\cdot;e^{-\varphi})$ is locally Lipschitz continuous.
We refer the reader to \cite{ChenZhang22}  for a systemic study of the $p$-Bergman theory.

Notice that, provided $\varphi$ is locally bounded from above, the infimum in \eqref{Cond1} is precisely the reciprocal of $B_{x+P,p}(x;e^{-\varphi})$.
In this appendix, we will prove some regularity results concerning $B_{\Omega,p}$ as the domain $\Omega$ varies. As an application, we obtain another solution to the conjecture of Inayama \cite{Inayama}.

\begin{proposition} \label{Prop:App1}
Let $\Omega\subset\mathbb{C}^n$ be an open set and $\varphi$ be a measurable function on $\Omega$ which is locally bounded from above. Let $\{\Omega_j\}_{j=1}^\infty$ be a sequence of open subsets of $\Omega$ so that $\Omega_j\subset\Omega_{j+1}$ and $\Omega=\cup_{j=1}^\infty\Omega_j$, then
$$ \lim_{j\to+\infty} B_{\Omega_j,p}(x;e^{-\varphi}) = B_{\Omega,p}(x;e^{-\varphi}), \quad \forall x\in\Omega. $$
\end{proposition}

\begin{proof}
The proof is the same as the classical case of $p=2$.
\end{proof}

\begin{proposition} \label{Prop:App2}
Let $\Omega\subset\mathbb{C}^n$ be an open set and $\varphi$ be a measurable function on $\Omega$ which is locally bounded from above. Let $P$ be a sufficiently small holomorphic cylinder so that $\Omega_P:=\{x\in\Omega:x+P\Subset\Omega\}$ is nonempty.

(i) $x\mapsto B_{x+P,p}(x;e^{-\varphi})$ is an upper semi-continuous function on $\Omega_P$, i.e.
$$ \varlimsup_{z\to x} B_{z+P,p}(z;e^{-\varphi}) \leqslant B_{x+P,p}(x;e^{-\varphi}), \quad \forall x\in\Omega_P. $$

(ii) for any $x\in\Omega_P$,
$$ \varliminf_{z\to x} B_{z+P,p}(z;e^{-\varphi}) \geqslant \lim_{s\to1^+} B_{x+sP,p}(x;e^{-\varphi}). $$
\end{proposition}

\begin{proof}
For simplicity, we denote $B_{x+P,p}(x;e^{-\varphi})$ by $\mathfrak{B}(x,P)$.

(i) Given $x\in\Omega_P$, we choose a sequence $\{x_j\}_{j=1}^\infty$ in $\Omega_P$ such that $x_j\to x$ and
$$ \varlimsup_{z\to x} \mathfrak{B}(z,P) = \lim_{j\to+\infty} \mathfrak{B}(x_j,P). $$
For each $j$, there exists an $f_j\in A^p(x_j+P;e^{-\varphi})$ such that $\int_{x_j+P}|f_j|^pe^{-\varphi}d\lambda \leqslant 1$ and $f_j(x_j) = \mathfrak{B}(x_j,P)^{1/p}$.
Let $s\in(0,1)$ be fixed for the moment, then $x+sP\Subset x_j+P$ for all $j\gg1$. By Montel's theorem, there exists a subsequence $\{f_{j_k}|_{x+sP}\}_{k=1}^\infty$ that converges uniformly on any compact subset of $x+sP$ to some $f\in\mathcal{O}(x+sP)$. By Fatou's lemma,
\begin{equation*}
	\int_{x+sP}|f|^pe^{-\varphi}d\lambda = \int_{x+sP}\lim_{k\to+\infty}|f_{j_k}|^pe^{-\varphi}d\lambda
	\leqslant \varliminf_{k\to+\infty} \int_{x+sP}|f_{j_k}|^pe^{-\varphi}d\lambda \leqslant 1.
\end{equation*}
Since $f_{j_k}$ converges compactly to $f$ and $x_j\to x$, it is clear that
$$ f(x) = \lim_{k\to+\infty} f_{j_k}(x_{j_k}) = \big( \varlimsup_{z\to x}\mathfrak{B}(z,P) \big)^{1/p}. $$
By definition,
$$ \mathfrak{B}(x,sP) \geqslant |f(x)|^p = \varlimsup_{z\to x} \mathfrak{B}(z,P).  $$
According to Proposition \ref{Prop:App1}, $\lim_{s\to1^-} \mathfrak{B}(x,sP) = \mathfrak{B}(x,P)$. Therefore,
$$ \varlimsup_{z\to x} \mathfrak{B}(z,P) \leqslant \mathfrak{B}(x,P). $$

(ii) Given $x\in\Omega_P$, we choose a sequence $\{x_j\}_{j=1}^\infty$ in $\Omega_P$ such that $x_j\to x$ and
$$ \varliminf_{z\to x} \mathfrak{B}(z,P) = \lim_{j\to+\infty} \mathfrak{B}(x_j,P). $$
We choose an $s>1$ with $x+sP\Subset\Omega$, then there exists an $f\in A^p(x+sP;e^{-\varphi})$ such that $\int_{x+sP}|f|^pe^{-\varphi}d\lambda\leqslant1$ and $\mathfrak{B}(x,sP)=|f(x)|^p$. If $j\gg1$, then $x_j+P\Subset x+sP$ and
$$ \mathfrak{B}(x_j,P) \geqslant \frac{|f(x_j)|^p}{\int_{x_j+P}|f|^pe^{-\varphi}d\lambda} \geqslant |f(x_j)|^p. $$
As a consequence,
\begin{equation*}
	\varliminf_{z\to x} \mathfrak{B}(z,P) = \lim_{j\to+\infty} \mathfrak{B}(x_j,P) \geqslant |f(x)|^p = \mathfrak{B}(x,sP).
\end{equation*}
Let $s\to1^+$, we complete the proof.
\end{proof}

\begin{proposition} \label{Prop:App3}
Let $\Omega\subset\mathbb{C}^n$ be an open set and $\varphi$ be a locally bounded measurable function on $\Omega$. For any constant $p>0$ and any holomorphic cylinder $P$, $x\mapsto B_{x+P,p}(x;e^{-\varphi})$ is a continuous function on $\Omega_P:=\{x\in\Omega:x+P\Subset\Omega\}$.
\end{proposition}

\begin{proof}
Let $\mathfrak{B}(x,P):=B_{x+P,p}(x;e^{-\varphi})$. Having Proposition \ref{Prop:App2}, it is sufficient to show that $\lim_{s\to1^+}\mathfrak{B}(x,sP)=\mathfrak{B}(x,P), \forall x\in\Omega_P$. \newpage

For simplicity, we assume that $x=0$. We choose a decreasing sequence $\{s_k\}_{k=0}^\infty$ satisfying $s_0P\Subset\Omega$ and $\lim_{k\to+\infty}s_k=1$. There exists an $f\in A^p(P;e^{-\varphi})$ such that $\int_P|f|^pe^{-\varphi}d\lambda \leqslant 1$ and $|f(0)|^p=\mathfrak{B}(0,P)$. Since $C:=\sup_{s_0P}|\varphi|<+\infty$, it is clear that $\int_P|f|^pd\lambda<+\infty$. For each $k\geqslant0$, we define a holomorphic function $f_k\in\mathcal{O}(s_kP)$ by $f_k(z):=f(z/s_k)$.

We fix an $\varepsilon>0$ for a moment. By the absolute continuity of Lebesgue integrals, there exists a $\delta>0$ so that $\int_E|f|^pd\lambda<\varepsilon$ for any measurable set $E\subset P$ with $\lambda(E)<\delta$. Since $\textbf{1}_{s_kP}f_k$ converges to $\textbf{1}_Pf$ almost everywhere, by Erogov's theorem, there exists a measurable set $E_\delta\subset s_0P$ with $\lambda(E_\delta)<\delta$ such that $\textbf{1}_{s_kP}|f_k|^pe^{-\varphi}$ converges uniformly to $\textbf{1}_P|f|^pe^{-\varphi}$ on $s_0P\setminus E_\delta$. By direct computations,
\begin{align*}
&\, \left| \int_P|f|^pe^{-\varphi}d\lambda - \int_{s_kP}|f_k|^pe^{-\varphi}d\lambda \right| \\
\leqslant &\, \int_{s_0P\setminus E_\delta} \big|\textbf{1}_P|f|^pe^{-\varphi} - \textbf{1}_{s_kP}|f_k|^pe^{-\varphi}\big| d\lambda + \int_{P\cap E_\delta}|f|^pe^{-\varphi}d\lambda + \int_{s_kP\cap E_\delta}|f_k|^pe^{-\varphi}d\lambda \\
\leqslant &\, \int_{s_0P\setminus E_\delta} \left|\cdots\right| d\lambda + e^C\int_{P\cap E_\delta}|f|^pd\lambda + s_k^{2n}e^C\int_{P\cap s_k^{-1}E_\delta}|f|^pd\lambda.
\end{align*}
Since $\lambda(s_k^{-1}E_\delta)\leqslant \lambda(E_\delta)<\delta$, it follows that
\begin{equation*}
\lim_{k\to+\infty} \left| \int_P|f|^pe^{-\varphi}d\lambda - \int_{s_kP}|f_k|^pe^{-\varphi}d\lambda \right| \leqslant 2e^C\varepsilon.
\end{equation*}
Since $\varepsilon>0$ is arbitrary, it follows that
\begin{equation*}
\lim_{k\to+\infty} \int_{s_kP}|f_k|^pe^{-\varphi}d\lambda = \int_P|f|^pe^{-\varphi}d\lambda \leqslant 1.
\end{equation*}
By the definition of $p$-Bergman kernels,
\begin{equation*}
\lim_{s\to1^+}\mathfrak{B}(0,sP) = \lim_{k\to+\infty}\mathfrak{B}(0,s_kP) \geqslant \lim_{k\to+\infty}\frac{|f_k(0)|^p}{\int_{s_kP}|f_k|^pe^{-\varphi}d\lambda} \geqslant \mathfrak{B}(0,P).
\end{equation*}
Since $\mathfrak{B}(0,sP)\leqslant\mathfrak{B}(0,P)$ for any $s>1$, this completes the proof.
\end{proof}

\begin{proposition} \label{Prop:App4}
Let $\Omega\subset\mathbb{C}^n$ be an open set and $\varphi$ be a psh function on $\Omega$. For any constant $0<p\leqslant2$ and any holomorphic cylinder $P$, $x\mapsto B_{x+P,p}(x;e^{-\varphi})$ is a continuous function on $\Omega_P:=\{x\in\Omega:x+P\Subset\Omega\}$.
\end{proposition}

\begin{proof}
Let $\mathfrak{B}(x,P):=B_{x+P,p}(x;e^{-\varphi})$. Having Proposition \ref{Prop:App2}, it is sufficient to show that $\lim_{s\to1^+}\mathfrak{B}(x,sP)=\mathfrak{B}(x,P), \forall x\in\Omega_P$.

For simplicity, we may assume that $x=0$, $P=\mathbb{D}_r\times\mathbb{B}_{r'}^{n-1}$ and $s_0P\Subset\Omega$ for some $s_0>1$. Let $\psi(z):=\max\{\log\frac{|z_1|}{r},\log\frac{|z'|}{r'}\}$, then $\psi$ is a psh function on $\mathbb{C}^n$ and $\{\psi<t\}=e^tP$ for any $t\in\mathbb{R}$. We consider a pseudoconvex domain
$$ \tilde{\Omega} := \{ (\tau,z)\in \mathbb{C}\times\mathbb{C}^n: \operatorname{Re}\tau<\log s_0, \psi(z)<\operatorname{Re}\tau \}. $$
Let $\pi:\tilde{\Omega}\to\mathbb{C}$ be the natural projection, then $\Omega_\tau:=\pi^{-1}(\tau) = e^{\operatorname{Re}\tau}P$ for all $\tau$. By the optimal $L^p$ extension theorem ($0<p\leqslant2$) and the Guan-Zhou method (see \cite{GuanZhou15} for details), one can shows that
$$ \tau\mapsto B_{\Omega_\tau,p}(0;e^{-\varphi}) $$
is a log-psh function on $\{\tau\in\mathbb{C}:\operatorname{Re}\tau<\log s_0\}$. Since $\Omega_\tau$ is independent of $\operatorname{Im}\tau$,
$$ \log B_{\Omega_t,p}(0;e^{-\varphi}) = \log \mathfrak{B}(0,e^tP) $$
is a convex function of $t\in(-\infty,\log s_0)$. In particular, $\lim_{s\to1}\mathfrak{B}(0,sP)=\mathfrak{B}(0,P)$. This completes the proof.
\end{proof}

Proposition \ref{Prop:App4} provides another solution to Inayama's conjecture.

\begin{theorem}
Let $\varphi$ be an upper semi-continuous function on $\Omega\subset\mathbb{C}^n$, then $\varphi$ is pluriharmonic if and only if $L_\varphi(x,P)\equiv1$ for all holomorphic cylinder $x+P\subset\Omega$.
\end{theorem}

\begin{proof}
We only need to prove the sufficiency.
Since $\varphi$ is upper semi-continuous and $L_\varphi\leqslant1$, we see that $\varphi$ satisfies the optimal $L^2$-extension condition. Consequently, $\varphi$ is a psh function. Let $P$ be a sufficiently small holomorphic cylinder in $\mathbb{C}^n$, then $L_\varphi\equiv1$ implies
$$  e^{\varphi(x)} = \textup{Vol}(P)\times B_{x+P,2}(x;e^{-\varphi}), \quad \forall x\in\Omega_P. $$
Then it follows from Proposition \ref{Prop:App4} that $\varphi$ is a continuous function on $\Omega_P$. Having Theorem \ref{MainThm:Inayama}, this completes the proof.
\end{proof}

\noindent\textbf{Acknowledgements.} The authors would like to thank Prof. Fusheng Deng, Prof. Zhiwei Wang, Prof. Xiangyu Zhou and Dr. Hui Yang for useful discussions. The second author also wants to thank Dr. Takahiro Inayama for sharing the conjecture \cite[Conjecture A.2]{Inayama} and his progress via email.

\end{document}